\documentclass{amsart}
\usepackage{amsmath}
\usepackage{amssymb}
\usepackage{mathtools}
\usepackage{extpfeil}
\usepackage{stmaryrd}
\usepackage{mathabx}
\usepackage[usenames]{color}
 
\usepackage{tikz}
\usetikzlibrary{matrix,arrows,positioning,calc}
\tikzset{hassenode/.style={shape=circle,fill=white, draw=black, scale=.4}}
\usepackage{mdwlist}
\usepackage{tikz-cd}
\newcommand{\Cnt}{\mathbf{Cnt}}
\newcommand{\CC}{\mathbf{CC}}
\newcommand{\CCf}{\mathbb{CC}}
\newcommand{\BCC}{\mathbf{BCC}}
\newcommand{\GMC}{\mathbf{GMC}}
\newcommand{\PV}{\mathbf{PV}}
\newcommand{\CR}{\mathbf{CR}}
\newcommand{\PVRM}{\mathbf{PVRM}}
\newcommand{\PVRMf}{\mathbb{PVRM}}
\newcommand{\2}{\mathbf{2}}

\newcommand{\G}{\mathbf{G}}
\newcommand{\Ap}{\mathbf{Ap}}

\newcommand{\Com}{\mathbf{Com}}
\newcommand{\V}{\mathbf{V}}

\newcommand{\W}{\mathbf{W}}
\newcommand{\D}{\mathsf{D}}
\newcommand{\R}{\mathsf{R}}
\newcommand{\q}{\mathfrak{q}}
\newcommand{\smi}{\mathsf{smi}}
\newcommand{\mi}{\mathsf{mi}}
\newcommand{\fmi}{\mathsf{fmi}}
\newcommand{\sfmi}{\mathsf{sfmi}}
\newcommand{\sji}{\mathsf{sji}}
\newcommand{\ji}{\mathsf{ji}}
\newcommand{\fji}{\mathsf{fji}}
\newcommand{\sfji}{\mathsf{sfji}}
\newcommand{\id}{1}

\newcommand{\INF}{\mathbf{inf}}
\newcommand{\TOP}{\mathsf{T}}

\newcommand{\Z}{\mathbb{Z}}
\newcommand{\up}[1]{\textup{#1}}
\newcommand{\rapprox}[1]{\stackrel{\ref{#1}}{=}}
\newcommand{\feq}{=}
\newcommand{\rfeq}[1]{\stackrel{\ref{#1}}{=}}
\newcommand{\con}{\operatorname{con}}

\newcommand{\bp}{{\bf p}}
\newcommand{\bq}{{\bf q}}
\newcommand{\bu}{{\bf u}}
\newcommand{\bv}{{\bf v}}
\newcommand{\bw}{{\bf w}}
\newcommand{\bx}{{\bf x}}

\newcommand{\bz}{{\bf z}}

\theoremstyle{plain}
\newtheorem{theorem}{Theorem}[section]
\newtheorem{lemma}[theorem]{Lemma}
\newtheorem{fact}[theorem]{Fact}
\newtheorem{proposition}[theorem]{Proposition}
\newtheorem{corollary}[theorem]{Corollary}
\theoremstyle{definition}
\newtheorem{definition}[theorem]{Definition}
\newtheorem{example}[theorem]{Example}
\newtheorem{problem}[theorem]{Problem}
\newtheorem{caution}[theorem]{Caution}
\newtheorem{remark}[theorem]{Remark}
\newtheorem{principle}[theorem]{Principle}

\DeclareMathOperator{\Atoms}{\mathsf{Atoms}}
\DeclareMathOperator{\Excl}{\mathsf{Excl}}

\newcommand{\separator}{\noindent\line(1,0){100}}

\newcommand{\adjunction}[4]{
\begin{tikzpicture}
\matrix (m) [matrix of math nodes, row sep=4em,column sep=2em, text height=1.5ex, text depth=0.25ex,ampersand replacement=\&]
{ #1 \& \&  #3 \\};
\path[>=latex,->>]($ (m-1-3.west) + (0,3pt) $)  edge [bend right=15] node[above] {#2}  ($ (m-1-1.east) + (0,3.pt) $);
\path[>=latex,right hook->]  ($ (m-1-1.east)+ (0,-3pt) $) edge  node[below] {#4} ($(m-1-3.west)+ (0,-3pt) $)  ;
\end{tikzpicture}
}

\begin{document}


%
%

\title{On the atoms of algebraic lattices arising in $q$-theory}

\author[A. Egri-Nagy]{Attila Egri-Nagy}
\address{Centre for Research in Mathematics, School of Computing and Mathematics, Western Sydney University, Locked Bag 1797, Penrith NSW 2751, Australia}
\email{attila@egri-nagy.hu}

\author[M. Jackson]{Marcel Jackson$^\MakeLowercase{a}$}
\address{Department of Mathematics and Statistics, La Trobe University, Victoria~3086, Australia}
\email{m.g.jackson@latrobe.edu.au}

\author[J. Rhodes]{John Rhodes$^\MakeLowercase{b}$}
\address{Department of Mathematics, University of California, Berkeley, 970 Evans Hall \#3840, Berkeley, CA~94720, USA}
\email{rhodes@math.berkeley.edu}
\email{blvdbastille@aol.com}

\author[B. Steinberg]{Benjamin Steinberg$^\MakeLowercase{c}$}
\address{Department of Mathematics, City College of New York, NAC~8/133, Convent Ave at 138th Street, New York, NY~10031, USA}
\email{bsteinberg@ccny.cuny.edu}
\thanks{$^a$Marcel Jackson was supported by ARC Discovery Project DP1094578 and Future Fellowship FT120100666}
\thanks{$^b$John Rhodes thanks the Simons Foundation Collaboration Grants for Mathematicians for travel grant \#313548}
\thanks{$^c$Benjamin Steinberg was supported in part by NSERC, a grant from the Simons Foundation (\#245268 to Benjamin Steinberg) and the Binational Science Foundation of Israel and the US (\#2012080) and by a PSC-CUNY grant. Some of this work was performed while Benjamin Steinberg was at Carleton University.}


\begin{abstract}
We determine many of the atoms of the algebraic lattices arising in $\mathfrak{q}$-theory of finite semigroups.
\end{abstract}

\maketitle

\section{Introduction}

All undefined terminology is given in \cite[Chapter~2]{qtheor} with which we assume the reader is familiar.

One way to view the $\q$-theory of finite semigroups is by analogy with the real analysis theory of continuous and differentiable functions from $[0,1]$ to itself.
The analogy is given by replacing $[0,1]$ with the complete algebraic lattice $\PV$ of all pseudovarieties of finite semigroups, replacing continuous functions with $\Cnt(\PV)$, and replacing differentiable functions with $\GMC(\PV)$; see \cite[Chapter~2]{qtheor}.

The collections of relational morphisms $\in \bf CC$ ($\bf PVRM$) give ``coordinates" (closely related to the graph of the continuous function given by applying the $\q$~operator) which, on application of~$\q$, yields, $\CC\q = \Cnt(\PV)$ and $\PVRM\q = \GMC(\PV)$.

For any $\mathsf{X} \subseteq \Cnt(\PV)$, let $\mathsf{X}^+$ denote the members~$\alpha$ of~$\mathsf{X}$ such that $\alpha(\V) \supseteq \V$ for all $\V \in \PV$.
Similarly, let $\mathsf{X}^-$ denote the members~$\beta$ of~$\mathsf{X}$ such that $\beta(\V) \subseteq \V$ for all $\V \in \PV$.

Next, $\CC$, $\CC^+$, $\CC^-$, $\PVRM$, $\PVRM^+$, and $\PVRM^-$ are defined so that $\CC \q = \Cnt(\PV)$, $\CC^+ \q = \Cnt(\PV)^+$, and so on.

Now since $\Cnt(\PV)$, $\Cnt(\PV)^+$, $\Cnt(\PV)^-$, $\GMC(\PV)$, $\GMC(\PV)^+$, and $\GMC(\PV)^-$ are all complete algebraic lattices, a natural question to ask is \textit{what are their atoms}?
Also we ask the same question for the complete algebraic lattices $\CC$, $\CC^+$, $\CC^-$, $\PVRM$, $\PVRM^+$, $\PVRM^-$, etc. including some minor variations of these.

We make significant progress on answering these questions; see Figures~\ref{tab:results1} and~\ref{tab:results2}.

So what are the methods of proofs?
For those having no atoms we use the obvious Principle~\ref{principle:1.7}.
For others we use the many Galois connections stemming from $\q$-theory \cite[Chapter~2]{qtheor} and then apply Proposition~\ref{prop:mqadjunction}.  
In determining the atoms of $\GMC$ and $\GMC^-$ we need to know which of the well-known atoms of $\PV$ (see \cite[Table~7.1]{qtheor}) lift, are projective, or are very small; see Definition~\ref{def:lifts}.
We determine, for each atom of $\PV$, when each of these properties hold; see Theorems~\ref{thm:1.24} and~\ref{thm:132}.

A big surprise arose when the $\Atoms(\Cnt(\PV)^+)$ turned out to be in one-to-one correspondence with the compact $\smi$ elements of $\PV$, where the compact elements of $\PV$ are the pseudovarieties generated by
 a single finite semigroup $S$; see Section~\ref{sec:irreducibility} and Theorem~\ref{thm:CntPVplusatoms}, Fact~\ref{fact:compactsmi}, and Remark~\ref{remark:compactsmi} for definitions and elementary properties.
Then the question arises: are there any compact $\smi$ pseudovarieties?
We prove that an infinite number exist.
To do this we first identify some basic syntactic conditions on an equation that guarantee it defines a $\smi$ pseudovariety (Proposition~\ref{prop:133}).  While these are not in general compact (Propositions \ref{prop:balancedcase} and \ref{prop:unbalancedcase}) we find two infinite families that are; see Section \ref{sec:compact}.  The method in each case is to show that there is a semigroup $S$ in the pseudovariety with the property that any equation \emph{not} following from the defining ones can be found to fail on $S$.  This semigroup $S$ generates the pseudovariety.

We conclude the article with two main problems and some other associated unresolved questions relating to compact $\smi$ pseudovarieties.


\section{Preliminaries}
Here we give few essential definitions, but making the paper self-contained would render the paper unreasonably long. Any undefined terminology can be found in~\cite[Chapter~2]{qtheor}, which we suggest that the reader keeps handy.  We follow the convention there that homomorphisms are written on the right of their arguments, but continuous operators on a lattice are written on the left.  A mapping of complete lattices is said to be \textbf{sup} if it preserves all suprema and \textbf{inf} if it preserves all infinima.

\subsection{Algebraic lattices} An element of a lattice is \emph{compact} if whenever it is less than or equal to the join of
a collection of elements, then it is actually below the join of a finite
subcollection. A complete lattice is \emph{algebraic} if each element is a join of compact elements. The set of compact elements of an algebraic lattice $L$ is denoted by $K(L)$.  The principal ideal generated by $\ell\in L$  is denoted by $\ell^{\downarrow}$.  The bottom and top of a lattice will be denoted by $\mathsf B$ and $\mathsf T$, respectively.

\subsection{Relational morphisms}

Let $S$ and $T$ be semigroups then a \emph{relational morphism} $\varphi\colon S\rightarrow T$ is a function $\varphi\colon S\rightarrow 2^T$ such that $s\varphi\neq\emptyset$ and $s_1\varphi s_2\varphi\subseteq (s_1s_2)\varphi$ for all $s,s_1,s_2\in S$. Thus relational morphisms of semigroups are generalizations of semigroup homomorphisms: they are relations with morphic properties.

We denote by $\PV$ the algebraic lattice of pseudovarieties of finite semigroups and by $\Cnt(\PV)$ the monoid of all continuous self-maps of $\PV$.   A mapping $\alpha\colon L\to L$ on a lattice is \emph{continuous} if it is order preserving and commutes with directed joins. Note that $\Cnt(\PV)$ is an algebraic lattice with respect to the pointwise ordering where joins and finite meets are computed pointwise, but infinite meets are not pointwise!  The submonoid $\Cnt(\PV)^+$ consists of those continuous operators $\alpha$ satisfying $\mathbf V\leq \alpha(\mathbf V)$ for all pseudovarieties $\mathbf V$. See~\cite[Chapter~2.2]{qtheor}.

Denote by $\CC$ the algebraic lattice of all continuously closed classes of relational morphisms.  See~\cite[Definition~2.1.2]{qtheor} for the axiomatic definition of a continuously closed class.  The algebraic lattice of pseudovarieties of relational morphisms is denoted by $\PVRM$.  See~\cite[Definition~2.1.5]{qtheor} for the definition. The algebraic lattices $\CC^+$ and $\PVRM^+$ consist of those continuously closed classes, respectively pseudovarieties of relational morphisms, that contain all identity mappings.  See~\cite[Definitions~2.1.3 and~2.1.6]{qtheor}.

If $T$ is a finite semigroup, we denote by $(T)$ the pseudovariety generated by $T$.  Similarly, if $f$ is a relational morphism, then $(f)$ denotes the pseudovariety of relational morphisms generated by $f$.

\subsection{The $\q$-operator}

If $\mathsf R$ is a continuously closed class, then $\mathsf R\q$ is the continuous operator on $\PV$ given by $\mathsf R\q(\mathbf V)$ is the pseudovariety of all semigroups $S$ such that there is a relational morphism $f\colon S\to T$ with $f\in \mathsf R$ and $T\in \V$.
The operator $\q\colon \CC\to \Cnt(\PV)$ in surjective, order preserving and continuous.  It preserves finite infima and all joins.  Moreover, it has a section $M\colon \Cnt(\PV)\to \CC$ given by \[M(\alpha)=\{f\colon S\to T\mid S\in \alpha((T))\}.\]  One has that $M(\alpha)$ is the unique maximum element of $\CC$ mapping to $\alpha$ under $\q$.  See~\cite[Chapter~2.3]{qtheor} for details.  The mapping $\q$ takes $\PVRM$ to the collection $\GMC(\PV)$ of all continuous operators satisfying the generalized Malcev condition~\cite[Definition~2.3.21]{qtheor}. The mapping $\q\colon \PVRM\to \GMC(\PV)$ preserves all sups and infs and has sections $\max$ and $\min$ taking each operator in $\GMC$ to the unique maximum, respectively minimum, pseudovariety of relational morphisms giving rise to it.  See~\cite[Chapter~2.3.2]{qtheor} for details.

\subsection{Irreducibility} \label{sec:irreducibility}
The following notions are defined in~\cite[Chapter~6.1.2]{qtheor}.  An element $\ell$ in a lattice $L$ is \emph{meet irreducible} $\mi$ if $\ell\geq \bigwedge X$ implies $\ell\geq x$ for some $x\in X$.  It is \emph{strictly meet irreducible} if $\ell=\bigwedge X$ implies $\ell=x$ for some $x\in X$.  We write $\fmi$, respectively,  $\sfmi$ for the analogous properties when $X$ is constrained to be finite.  The dual notions for joins are denoted $\ji$, $\sji$, $\fji$ and $\sfji$.  So, for example, $\ell$ is $\ji$ if $\ell\leq \bigvee X$ implies that $\ell\leq x$ for some $x\in X$.  Note that in an algebraic lattice, a $\ji$ element must be compact and, in fact, the $\ji$ elements are precisely the $\fji$ compact elements.

\section{Atoms}
An \emph{atom} of a lattice $L$ is a  cover of the bottom $\mathsf B$, that is, a minimal element of $L\setminus \{\mathsf B\}$.


The following fact is well known and can be found as~\cite[Lemma~4.49]{McKenzie}.
\begin{fact}\label{fact:1.1}
If $L$ is an algebraic lattice and $\ell_1,\ell_2\in L$, $\ell_1\leq\ell_2$, then $[\ell_1,\ell_2]$ is an algebraic lattice with compact elements $(K(L)\cap\ell_2^\downarrow) \vee  \ell_1$.
\end{fact}
%
%
%

\begin{corollary}
The compact elements of $[\mathsf{B},\ell_2]$ equal $K(L)\cap\ell_2^\downarrow$.
\end{corollary}

\begin{fact}
$\Atoms\big([\mathsf{B},\ell_2]\big)=\Atoms(L)\cap\ell_2^\downarrow$. Atoms are compact and $\sji$ in algebraic lattices.
\label{fact:1.3}
\end{fact}
\begin{proof}
The first statement is clear.  In an algebraic lattice $L$, any $\sji$ element is compact as it is a join of compact elements.  Atoms are clearly $\sji$ because only the bottom is strictly below them.
\end{proof}

\begin{caution}
In an algebraic lattice $L$, $\Atoms(L)$ can be empty.
\end{caution}

\begin{corollary}
 If $L$ has no atoms, then $[\mathsf{B},\ell_2]$ has no atoms.
\label{cor:intervalhasnoatoms}
\end{corollary}

\begin{remark}
In an algebraic lattice $L$, the atoms of $[\ell_1,\mathsf{T}]$ are the covers of $\ell_1$ in $L$, so in general, they are unrelated to $\Atoms(L)$.
\end{remark}

\begin{principle}[No atoms for $L$, an algebraic lattice] \label{principle:1.7}
If each compact element $c\neq\mathsf{B}$ has a compact element other than $\mathsf B$ strictly below, then $\Atoms(L)=\varnothing$, and conversely, since the atoms are the compact covers.
\end{principle}





Principle~\ref{principle:1.7} was used in~\cite[Proposition 7.1.24]{qtheor} to prove the following proposition.
\begin{proposition} The algebraic lattice $\Cnt(\PV)$ has no atoms.
\label{prop:CntPV_has_no_atoms}
\end{proposition}

As a consequence, we can prove that $\CC$ has no atoms.

\begin{proposition}
The algebraic lattice $\CC$ has no atoms.
\label{prop:CC_has_no_atoms}
\end{proposition}
\begin{proof}
By~\cite[Theorem~2.3.9]{qtheor}, there is a surjective map $\mathfrak q\colon \CC\to \Cnt(\PV)$ preserving all sups and finite meets. The bottom of $\Cnt(\PV)$ is the constant map to the trivial pseudovariety. In~\cite[Page~121]{qtheor} it is shown that each constant map has a unique preimage under $\mathfrak q$, hence if $\mathsf R$ is not the bottom of $\CC$, then is does not map to the bottom of $\Cnt(\PV)$ under $\mathfrak q$. Since $\Cnt(\PV)$ has no atoms, we can find $\mathsf B\neq \alpha < R\mathfrak q$.  By surjectivity, there exists $\mathsf S$ with $\mathsf S\mathfrak q=\alpha$.  Since $\mathfrak q$ preserves finite infs, we obtain $(\mathsf R\cap \mathsf S)\mathfrak q=\alpha$ and so $\mathsf R\cap \mathsf S<\mathsf R$ and $\mathsf R\cap \mathsf S$ is not the bottom.  Thus $\CC$ has no atoms.
\end{proof}

The reader is referred to~\cite[Proposition~2.1.11]{qtheor} for the definition of $\CCf$ and~\cite[Page~75]{qtheor} for the definition of $\PVRMf$.

\begin{fact} If $\D$ denotes the class of all divisions, then
\begin{enumerate}
\item $\CCf(\id_\V\mid \V\in\PV)=\D$
\item  $\PVRMf(\id_\V\mid \V\in\PV)=\D$
\end{enumerate}
\end{fact}
\begin{proof}
One way of calculating $\CCf$ is
$$\CCf(X)=\{ f\mid f\subseteq_s d_1(f_1\times\cdots\times f_n)d_2,\quad d_1,d_2\in\D, f_i\in X\}$$
(see Proposition 2.1.14 in \cite{qtheor}) from which (1) follows.
 Also  $\CCf(\id_\V\mid \V\in\PV)$ is closed under Axiom (co-re),  as $\D$ is, so is a pseudovariety of relational morphisms in $\PVRM$ (see Proposition 2.1.8(c) in \cite{qtheor}) proving (2).
\end{proof}

\begin{definition}
\begin{eqnarray*}
\Cnt(\PV)^-&=&\{\alpha\in\Cnt(\PV)\mid \alpha\leq 1_{\PV}\}\\
\CC^-&=&\{\beta\in\CC\mid \beta\leq \D \}
\end{eqnarray*}
\label{def:CntPVminusCCminus}
\end{definition}

\begin{fact} $(\CC^-)\q=\Cnt(\PV)^-$
\label{fact:CCminusq_equals_CntPVminus} 
\end{fact}
\begin{proof}
Since $\D\q=\id_{\PV}$ and $\q$ is order preserving  $(\CC^-)\q\subseteq \Cnt(\PV)^-$. If $\alpha\in\Cnt(\PV)^-$, then since $\alpha\in\Cnt(\PV)$, there exists $\R\in \CC$, so $\R\q=\alpha$.
Thus since $\q$ preserves finite intersections (intersection equals meet)
$(\R\cap \D)\q=\alpha$ and $\R\cap \D\in\CC^-$.
\end{proof}

\begin{corollary} 
The lattices $\CC, \CC^-, \Cnt(\PV), \Cnt(\PV)^-$ have no atoms.
\label{cor:NoAtoms}
\end{corollary}
\begin{proof}
Use Proposition~\ref{prop:CntPV_has_no_atoms}, Proposition~\ref{prop:CC_has_no_atoms}, Corollary~\ref{cor:intervalhasnoatoms} together with Definition~\ref{def:CntPVminusCCminus}.
\end{proof}


In Section~\ref{sec:compact} we  describe two infinite families of compact $\smi$ elements of $\PV$.  The following theorem then shows that there are infinitely many atoms in $\Cnt(\PV)^+$.

\begin{theorem}\label{thm:CntPVplusatoms}
There is a bijection between $\Atoms(\Cnt(\PV)^+)$ and the compact $\smi$ elements in $\PV$.
\end{theorem}

Before the proof of Theorem~\ref{thm:CntPVplusatoms} we require the following fact.

\begin{fact}\label{fact:compactsmi} Let $T$ be a finite semigroup.
\begin{enumerate}
\item $(T)$ is  a compact $\smi$ of $\PV$, if and only if there exists a finite semigroup $S$ such that $(S)$ is the unique cover of $(T)$  in $\PV$ in the sense of $\bigwedge\{\W\in\PV\mid\W>(T)\}=(S)$.
\item $(T)$ is a compact $\smi$  with unique cover compact $(S)$ in $\PV$ if and only if,  for all $\W\in\PV$, $\W>(T)$ implies $S\in W$.
\end{enumerate}
\end{fact}
\begin{proof} See~\cite[Proposition~7.1.13]{qtheor}.
\end{proof}

\begin{remark}[Paraphrasing Fact~\ref{fact:compactsmi}] \label{remark:compactsmi}
Compact $\smi$ pseudovarieties exist if and only if there exist finite semigroups $T,S$ such that $(T)<(S)$ and $\W\in\PV$, $\W>(T)$ if and only if $S\in\W$.
\end{remark}

We now prove Theorem~\ref{thm:CntPVplusatoms}. In the following we denote  $\big(\delta(S,T)\vee \id_{\PV}\big)\in\Cnt(\PV)^+$ by $P(S,T)$ where we recall that if $S,T$ are finite semigroups, then \[\delta(S,T)(\mathbf V) = \begin{cases} (S), \text{if}\ T\in \mathbf V\\ \mathsf B, \text{else.}\end{cases}\]  Every compact element of $\Cnt(\PV)$ is a finite join of elements of the form $\delta(S,T)$ and hence any compact element of $\Cnt(\PV)^+$ must be a finite join of elements of the form $P(S,T)$ by Fact~\ref{fact:1.1}. See~\cite[Proposition~2.2.2]{qtheor} for details.  Consequently, an atom of $\Cnt(\PV)^+$ must be of the form $P(S,T)$ for some semigroups $S,T$.

\begin{proof}[Proof of Theorem~\ref{thm:CntPVplusatoms}]
Let $(T)$ be a compact $\smi$ with a unique cover $(S)$.  We prove that $P(S,T)$ is an atom of $\Cnt(\PV)^+$.
Then, by Fact~\ref{fact:compactsmi} restricted to compact $(S_1)$, we have
\[P(S,T)((S_1)) = \begin{cases} (S), & \text{if}\ (S_1)=(T)\\ (S_1), & \text{else}\end{cases}\]
because if $T\notin (S_1)$, then $(S_1)\mapsto (S_1)$. If $T\in (S_1),\ (T)<(S_1)$, then $(S_1)\mapsto (S_1)\vee (S)=(S_1)$ by \eqref{fact:compactsmi}.
Clearly, this is an atom of $\Cnt(\PV)^+$, (since $\id_\PV\leq\alpha<P(S,T)$ implies $\alpha=\id_\PV$).

Next suppose that $\alpha$ is an atom of $\Cnt(\PV)^+$.  We already observed that $\alpha=P(S,T)$ for some finite semigroups $S,T$.
Clearly $P(S,T)\neq \id_\PV$ if and only if $(S)\nleq (T)$ and so we must have
\begin{align*}
(T)&\mapsto(T)\vee(S)>(T)\\
(T)\in\W&\mapsto\W\vee(S)\\
(T)\notin\W&\mapsto\W
\end{align*}
We claim that $(T)$ is a compact $\smi$ with unique cover $(S)$.  Assume otherwise.
Choose a finite semigroup $T_1$ so $(T)<(T_1)$ and $S\notin(T_1)$  (cf.~Fact~\ref{fact:compactsmi}).
Then $P(S,T_1)< P(S,T)$ ($\leq$ is clear and $(T)\mapsto(T)$ in the first case, $(T)\mapsto(S)\vee(T)\neq(T)$ in the second case).
Thus $a$ is an atom of $\Cnt(\PV)^+$ if and only if $a=P(S,T)$ with $(T)$ a compact $\smi$ in $\PV$ with unique cover $(S)$.  This establishes the bijection between atoms of $\Cnt(\PV)^+$ and compact $\smi$s.
\end{proof}


\begin{definition}\label{def:extendingQBook}
Let $\Cnt(\PV)^{-}$ consist of those operators $\alpha$ with $\alpha(\mathbf V)\leq \mathbf V$ and put $\GMC^-(\PV)=\GMC(\PV)\cap \Cnt(\PV)^-$.
Notice that by~\cite[Corollary~3.5.22]{qtheor} $\D=\min(\id_\PV)$ but $\D\neq \max(\id_{\PV})$ by~\cite[Example~2.4.1]{qtheor}, so this motivates the following new extended definition:
$\PVRM^{(+)}=[\max(\id_\PV),\TOP]<\PVRM^+$ (where $\PVRM^+=[\D,\PVRM]$).

Recall that \[\max(\id_\PV)=\{f\colon S\rightarrow T,\ \text{a relational morphism}\mid W\leq T, Wf^{-1}\in(W)\}.\] See~\cite[Proposition~2.3.32]{qtheor}. See, for example \eqref{eq:22_4} for why we define $\PVRM^{(+)}$.  Define
$\PVRM^-=[\id,\D]$.

Similarly, we define $\CC^+=[\D,\mathsf T]$, $\CC^{(+)}=[M(\id_\PV),\TOP]$, $\CC^-=[\mathsf B,\D]$ where these intervals are in the lattice $\CC$.  Similar definitions are used for $\BCC^\epsilon$, $\epsilon\in\{+, (+),-\}$, such as $\BCC\cap\CC^\epsilon$. The reader is referred to~\cite[Section~2.3.3]{qtheor} for $\BCC$ (the lattice of Birkhoff continuously closed classes) and the following facts.  It turns out that $M\colon \Cnt(\PV)\to \CC$ takes values in $\BCC$ and that each continuous operator $\alpha\in \Cnt(\PV)$ is the image of a unique minimum Birkhoff continuous closed class $m(\alpha)$.
We recall that \[M(\id_\PV)=\{f\colon S\rightarrow T\mid S\in (T)\}\] by~\cite[Equation~(2.15), page~63]{qtheor}.
\end{definition}

The following proposition will be useful in computing atoms.

\begin{proposition}
\label{prop:mqadjunction} 
Let $L_1$, $L_2$ be complete lattices.  The following hypothesis is denoted Hypothesis~\eqref{prop:mqadjunction}:
\begin{enumerate}
\item there is an adjunction
\[\adjunction{L_1}{$q$}{L_2}{$m$}\]
that is, $m$ is injective and $\mathbf{sup}$, $q$ is $\mathbf{inf}$ and onto;
\item  $B_2q^{-1}=B_1$ where $B_i$ is the bottom of $L_i$.
\end{enumerate}
Under Hypothesis~\eqref{prop:mqadjunction}, one has the conclusion:
\begin{itemize}
\item [(a)] $\Atoms(L_1)m=\Atoms(L_2)$;
\item [(b)]  $\Atoms(L_1)=\Atoms(L_2)q$
\end{itemize}
\end{proposition}

Before giving the proof of Proposition~\ref{prop:mqadjunction} we give an example and a counterexample.
\begin{example}
\label{exmp:qmadjunction}
\begin{enumerate}
\item An example  is~\cite[(2.34), Page~76]{qtheor}
\[\adjunction{\GMC}{$\q$}{\PVRM}{$\min$}.\] This satisfies Hypothesis~\eqref{prop:mqadjunction} because the bottom of of $\GMC$ is $C_{\{1\}}$, the constant map on $\PV$ always $\{1\}$ and the bottom of $\PVRM$ is \[\{\widetilde{1}\}=\{f\colon \{1\}\rightarrow T\mid f\ \text{is a relational morphism}\}.\] Then Hypothesis~\eqref{prop:mqadjunction} is satisfied as it is proved on~\cite[Page~121]{qtheor}.
\item Counterexample.
\begin{center}
\begin{tikzpicture}
\matrix [column sep=25pt,row sep=25pt]
{
\node[style=hassenode] (n2) {};\\
\node[style=hassenode] (n1) {}; \\
\node[style=hassenode] (n0) {}; \\
\node (label) {$L_1$}; \\
};
\tikzset{every node/.style={}}
\draw (n2) node [right] {$2$};
\draw (n1) node [right] {$1$};
\draw (n0) node [right] {$0\equiv$Bottom};
\draw (n2) -- (n1) -- (n0);
\end{tikzpicture}
\hskip20pt
\begin{tikzpicture}
\matrix [column sep=25pt,row sep=25pt]
{
\node[style=hassenode] (n2) {};\\
\node[style=hassenode] (n1) {}; \\
\node[style=hassenode] (nx) {}; \\
\node[style=hassenode] (n0) {}; \\
\node (label) {$L_2$}; \\
};
\draw (n2) node [right] {$2$};
\draw (n1) node [right] {$1$};
\draw (nx) node [right] {$x$};
\draw (n0) node [right] {$0\equiv$Bottom};
\draw (n2) -- (n1) -- (nx) -- (n0);
\end{tikzpicture}
\end{center}
Let $m\colon L_1\rightarrow L_2$ with $(j)m=j$ for $j=2,1,0$.  Let
$q\colon L_2\rightarrow L_1$ with $(j)q=j$ for $j=2,1,0$ and $(x)q=0$.
$\Atoms(L_1)=\{1\}$, $\Atoms(L_2)=\{x\}$, Proposition~\ref{prop:mqadjunction}(a),(b) are false, and (2) of Hypothesis~\eqref{prop:mqadjunction} fails.  Thus Hypothesis~\eqref{prop:mqadjunction} is necessary to imply Proposition~\ref{prop:mqadjunction}(a) or Proposition~\ref{prop:mqadjunction}(b).
\end{enumerate}
\end{example}

We now prove Proposition~\ref{prop:mqadjunction}.
\begin{proof}[Proof of~Proposition~\ref{prop:mqadjunction}]
Let us begin with the proof of (a).
If $a\in L_1$ is an atom and $am$ is not an atom of $L_2$, then there exists $\ell_2\in L_2$ such that $am>\ell_2>B_2$. Applying $q$ yields $amq=a>\ell_2q>B_1$ with $amq=a>\ell_2q$ following from the definition of $m$ and (1) of Hypothesis~\eqref{prop:mqadjunction} (cf.~\cite[Proposition~1.1.7]{qtheor}) and $\ell_2q>B_1$ by (2) of Hypothesis~\eqref{prop:mqadjunction}. But this contradicts that $a$ is an atom of $L_1$.

Conversely, let $A$ be an atom of $L_2$. Then $A q m=A\neq B_2$ since otherwise $A>A q m$ by (2) of Hypothesis~\eqref{prop:mqadjunction}, so $A>A q m>B_2$ by Proposition~\ref{prop:mqadjunction}(2), contradicting that $A$ is an atom of $L_2$.
Thus $A q m=A\neq B_2$. But $Aq$ is an atom of $L_1$, for if not there exists $C\in L_1$ with $Aq> C > B_1$. Applying $m$, which is injective and order preserving, yields $B_2=B_1m<Cm<Aqm=A$.
contradicting that~$A$ is an atom of $L_2$. This proves Proposition~\ref{prop:mqadjunction}(a).

To prove  Proposition~\ref{prop:mqadjunction}(b), just apply $q$ to both sides of  Proposition~\ref{prop:mqadjunction}(a). This completes the proof of  Proposition~\ref{prop:mqadjunction}.
\end{proof}

\subsection{Applications of Proposition~\ref{prop:mqadjunction}}
\begin{subequations}
\renewcommand{\theequation}{\thesection.\theparentequation)(\arabic{equation}}
\label{nn:qminadjunctions}
We use the $\q$ operator
 and the min map. Consider 
\begin{equation} \adjunction{\GMC}{$\q$}{\PVRM}{$\min$}. \end{equation}
  Hypothesis~\eqref{prop:mqadjunction} holds. See Example~\ref{exmp:qmadjunction}. Consider
\begin{equation} 
\adjunction{\GMC^-}{$\q$}{\PVRM^{-}}{$\min$}.
\end{equation}
Hypothesis~\eqref{prop:mqadjunction} holds. See Fact~\ref{fact:CCminusq_equals_CntPVminus} adapted to $\PVRM^-$ and  $\min$ is restriction of $\min$ from (\thesection.\ref{nn:qminadjunctions})(1). Consider
\begin{equation} 
\adjunction{\GMC^+}{$\q$}{\PVRM^{+}}{$\min$}.
\end{equation}
Then~(2) of  Hypothesis~\eqref{prop:mqadjunction} fails because $\min(\id_{\PV})=\D\neq\max(\id_{\PV})$.
So instead, consider
\begin{equation} 
\adjunction{\GMC^+}{$\q$}{\PVRM^{(+)}}{$\widetilde m$}
\label{eq:22_4}
\end{equation}
where $\widetilde m$ is the right adjoint of the restriction of $\q$.
 Hypothesis~\eqref{prop:mqadjunction} holds because of the way $\PVRM^{(+)}$ was defined (see Definition~\ref{def:extendingQBook}). Since the bottom of  $\PVRM^{(+)}$ is the maximal preimage of the identity map in $\PVRM$, (2) of Hypothesis~\eqref{prop:mqadjunction} holds. Consider
\begin{equation} 
\adjunction{\Cnt(\PV)}{$\q$}{\BCC}{$m$}.
\end{equation}
 Hypothesis~\eqref{prop:mqadjunction} holds. 
\begin{equation} 
\adjunction{\Cnt(\PV)^-}{$\q$}{\BCC^-}{$m$}.
\end{equation}
Hypothesis~\eqref{prop:mqadjunction} holds. Use (\thesection.\ref{nn:qminadjunctions})(5).
\begin{equation} 
\adjunction{\Cnt(\PV)^+}{$\q$}{\BCC^+}{$\min=m$}
\end{equation}
 Hypothesis~\eqref{prop:mqadjunction} fails, similar to (\thesection.\ref{nn:qminadjunctions})(3). Consider
\begin{equation} 
\adjunction{\Cnt(\PV)^+}{$\q$}{\BCC^{(+)}}{$\min=m$}
\end{equation}
 Hypothesis~\eqref{prop:mqadjunction} holds, similar to (\thesection.\ref{nn:qminadjunctions})(4). Note that
\begin{equation} 
\adjunction{\Cnt(\PV)}{$\q$}{\CC}{$m$}
\end{equation}
 \emph{is not defined} since $\q$ is \emph{not} $\INF$ on $\CC$. See \cite[Example~2.3.12]{qtheor}. Also
\begin{equation} 
\adjunction{\Cnt(\PV)}{$\q$}{\CC^{+}}{$m$}
\end{equation}
 \emph{is not defined} since $\q$ is \emph{not} $\INF$ on $\CC^+$. See \cite[Example 2.3.14]{qtheor}. Also 
\begin{equation} 
\adjunction{\Cnt(\PV)}{$\q$}{\CC^{(+)}}{$m$}
\end{equation}
 \emph{is not defined} since $\q$ is \emph{not} $\INF$ on $\CC^{(+)}$. This is similar proof as for (\thesection.\ref{nn:qminadjunctions})(10).
The details go as follows. $\CC^{(+)}$ is by Definition~\ref{def:extendingQBook} the interval $[M(\id_\PV), \mathsf T]$ in $\CC$ where \[M(\id_\PV)=\{f\colon S\rightarrow T\mid f \text{ is a}\text{ relational }\text{morphism and }(S)\leq(T)\}.\] Now \cite[Lemma~2.3.13]{qtheor} holds with the same proof if ``positive continuous closed class'' is changed to ``continuously closed class containing $M(\id_\PV)$'' and ``division'' is changed to ``member of $M(\id_\PV)$''. Now the proof of Example 2.3.14 goes through with the above changes.

Also,
\begin{equation} 
\adjunction{\Cnt(\PV)^-}{$\q$}{\CC^{-}}{$\min=m$}
\end{equation}
 \emph{is not defined} since $\q$ is \emph{not} $\INF$. Indeed, use \cite[Lemma~2.3.11]{qtheor} 
and follow the proof scheme of Example 2.3.12, but change the definition of $\#u_n$ as follows: choose $1\neq S\xhookrightarrow{\ j\ } T$, $(S)\leq (T)$, $u_n\colon S\hookrightarrow T^n$ $n\geq 1$ $k\mapsto(k,\dots,k)\equiv((k)j,\dots,(k)j)$.  Then $u_n$ is a division and $R_n=\CCf(u_n)$.
Now the proof follows as in Example 2.3.12.
\end{subequations}

The results so far are in Figure~\ref{tab:results1}.

\begin{figure}
\begin{tabular}{|c|p{2.05cm}|p{2.92cm}|p{2.65cm}|p{2.95cm}|}
\hline
& Subset of $\CC$ & Its atoms & Subset $\Cnt(\PV)$ & Its atoms\\
\hline
\hline
1 & $\CC$ & $\varnothing$ Prop~\ref{prop:CC_has_no_atoms} & $\Cnt(\PV)$ & $\varnothing$ Prop~\ref{prop:CntPV_has_no_atoms}\\
\hline
2 & $\CC^-$ & $\varnothing$ Cor~\ref{cor:NoAtoms} & $\Cnt(\PV)^-$ & $\varnothing$ Cor~\ref{cor:NoAtoms}\\
\hline
3 & $\BCC^{(+)}$ & $Xm$, $X$ is an atom of $\Cnt(\PV)^+$ & $\Cnt(\PV)^+$ & atoms are in one-to-one correspondence with compact $\smi$'s of $\PV$\\
\hline
4 & $\CC^+$ & ? wild guess $=\varnothing$  & & \\
\hline
5 & $\CC^{(+)}$ & ? wild guess $=\varnothing$  & & \\
\hline
6 & $\PVRM$ & $\Atoms(\GMC)\min$  &$\GMC$ & $\Atoms(\PVRM)\q$\\
\hline
7 & $\PVRM^-$ & $\Atoms(\GMC^-)\min$  &$\GMC^-$ & $\Atoms(\PVRM^-)\q$\\
\hline
8 & $\PVRM^{(+)}$ & $\Atoms(\GMC^+)\min$  &$\GMC^+$ & $\Atoms(\PVRM^{(+)})\q$\\
\hline
9 & $\PVRM^+$ & ?   & & \\
\hline
10 & $\BCC^{+}$ & ?   & & \\
\hline
11 & $\BCC$ & $\varnothing$   & $\Cnt(\PV)$ & $\varnothing$\\
\hline
12 & $\BCC^-$ & $\varnothing$   & $\Cnt(\PV)^-$ & $\varnothing$\\
\hline
\end{tabular}
\caption{Tabulation of results so far. See also later table Figure~\ref{tab:results2}.}
\label{tab:results1}
\end{figure}


\section{Atoms of $\PVRM$ and $\smi$ pseudovarieties}
The atoms of $\PV$ are the pseudovarieties generated by the two-element semigroups and by the cyclic groups of prime order. The notations are $\mathbf 2^r$ (the two-element right zero semigroup), $\mathbf 2^l$ (the two-element left zero semigroup), $\{0,1\}$ under multiplication (the two-element semilattice) and $N_2$ the two-element null semigroup.  Sometimes we abuse the distinction between these semigroups and the pseudovariety they generate.

\begin{definition} \label{def:lifts}
\begin{itemize}
\item[(a)] A finite semigroup $T$ \emph{lifts} if the existence of a surjective morphism from a finite semigroup $S$,  $\varphi\colon S\twoheadrightarrow T$  implies there exists a subsemigroup $T'$ of $S$ so $T$ is isomorphic to $T'$, $T\cong T'\leq S$. So for any such surjective homomorphism onto $T$ there are isomorphic copies of $T$ in the preimage.
\item[(b)] A finite semigroup $T$ is \emph{projective} if it lifts and $\varphi$ restricted to $T'$, as above, is an isomorphism, so $\varphi(T')=T$. In other words given a surjective homomorphism $\varphi\colon S\to T$, there is a splitting homomorphism $\psi\colon T\rightarrow S$ such that $\psi\varphi=\id_T$.
\item[(c)] A finite nontrivial semigroup $T$ or pseudovariety $(T)$ is said to be \emph{very small} if, for all finite semigroups~$S$, the join $(S)\vee (T)=(S\times T)$ either covers or equals $(S)$.  In the lattice theory literature, one would say that $(T)$ has the covering property.
\end{itemize}
Intuitively, $T$ lifts if we can find it by going backwards on surjective morphisms but the isomorphic copies have nothing to do with the map.
If it turns out that one surjective morphism respects the isomorphic copies then $T$ is projective.
Clearly projective implies lifts. For instance, $\Z_{p^n}$ lifts for any prime $p$,  but is not projective.

The semigroup $N_2$ is very small but does not lift.  Any nontrivial semilattice is very small by~\cite[Theorem~2.4]{SLsmall}.
\end{definition}

We next work on the atoms of pseudovarieties of relational morphisms of $\PVRM$ and $\PVRM^-$.
This will be some work.
We first consider $\PVRM$ so $(X)$ denotes the member of $\PVRM$ generated by a set $X$ of relational morphisms.   First recall that if $\mathbf V$ is a pseudovariety of semigroups, then
\[\widetilde{\mathbf V}=\{f\colon S\to T\mid S\in \mathbf V\}\] is a pseudovariety of relational morphisms and it is the unique pseudovariety of relational morphism sent by $\q$ to the constant mapping with image $\mathbf V$.  See~\cite[Page~121]{qtheor}.

\begin{lemma}
If $(f)$ is an atom of $\PVRM$ where $f\colon S\to T$  with $S\neq 1$, then $(S)$ must be an atom of $\PV$.
\label{lem:1.23}
\end{lemma}
\begin{proof}
 If $(S)$ is not an atom of $\PV$ (and so $S\neq \{1\}$), then there exists an atom $(a)$ of $\PV$, such that $(a)<(S)$. Indeed, if $S$ is a finite semigroup not containing $(N_2)$, then $S$ is completely regular; it must have a single $\mathcal J$-class if it doesn't have  $\{0,1\}$ as a divisor; it must be a group if it also doesn't have $\mathbf 2^r$ and $\mathbf 2^l$ as a divisor and it must be a trivial group if it has no cyclic group of primer order as a divisor.

So for some $n\geq 1$,  there is a division $d\colon a\rightarrow S^n$. Consider $df^n$.  Then $df^n\in (f)\cap \widetilde{(a)}$, but $f\notin \widetilde{(a)}$.  Therefore, $\mathsf B<(f)\cap \widetilde{(a)}<(f)$ and so $f$ is not an atom.
\end{proof}

\begin{theorem}
If $a=\mathbf 2^r$, $\mathbf 2^l$ or $\{0,1\}$, i.e., is a projective 
 atom of $\PV$, then $(\id_a)$ is an atom of $\PVRM$.
\label{thm:1.24}
\end{theorem}
\begin{proof}
Since the divisions form a pseudovariety of relational morphisms, it follows that if $(a)\in \PV$ is one of the above projective atoms and $f\colon S\to T$ belongs $(\id_a)$, then $f$ is a division. By closure of pseudovarieties of relational morphisms under range extension and corestriction, we may assume it is the inverse of a surjective homomorphism.   Also, $(1_a)\subseteq \widetilde{(a)}$ and so $f\in \widetilde{(a)}$, whence $S\in (a)$.
If $S$ is trivial, then $f$ belongs to the bottom of $\PVRM$.  Otherwise, $a$ is a subsemigroup of $S$ by elementary properties of $(\mathbf 2^r)$, $(\mathbf 2^l)$ and $(\{0,1\})$. Since $f$ is the inverse of a surjective homomorphism and $a$ is projective we obtain that $\id_a$ divides $f$ via the diagram
\[\begin{tikzcd}
a\ar[hook]{r}\ar{d}[swap]{1_a} & S\ar{d}{f}\\
a\ar{r} & T
\end{tikzcd}\]
where the top arrow is the inclusion and the bottom arrow is a homomorphism splitting of $f^{-1}|_{af^{-1}}\colon af^{-1}\to a$.
\end{proof}

\begin{lemma}\label{nocyclicfront}
A relational morphism $f\colon \langle x\rangle\to T$ does not generate an atom for any non-trivial cyclic semigroup $\langle x\rangle$.
\end{lemma}
\begin{proof}
The relational morphism $f$ contains a relational morphism of the form
\begin{equation*}
\begin{tikzpicture}
\matrix (m) [matrix of math nodes, row sep=4em,column sep=4em, text height=1.5ex, text depth=0.25ex,ampersand replacement=\&]
{ \langle y \rangle \& \langle z \rangle  \\
  \langle x \rangle \& \\ };
\path[>=latex,->>] (m-1-1.east)  edge node[above] {$b$}   (m-1-2.west);
\path[>=latex,->>,dashed] (m-1-1.south)  edge node[left] {$a$}  (m-2-1.north);
\end{tikzpicture}
\end{equation*}
where $y$ maps to $x$ under $a$ and to $z$ under $b$ and no proper subsemigroup of $\langle y\rangle$ maps onto $\langle x\rangle$ by $a$. It suffices to show that $a^{-1}b$ does not generate an atom and so we may assume that $f=a^{-1}b$.
Note that $a^{-1}$ is in $(a^{-1}bb^{-1})=(fb^{-1})$ which is contained in $(f)$ by closure of pseudovarieties of relational morphisms under codomain division.  Thus we may assume $f=a^{-1}$.

Non-trivial cyclic semigroups are not projective (one can verify this directly or use the results of either~\cite{projective} or~\cite{Zelenyuk} which imply that any projective finite semigroup is a band).
So there exists a surjective homomorphism $c\colon \langle u\rangle\twoheadrightarrow \langle y\rangle$ that does not split (using non-trivial cyclic semigroups are not projective) and, moreover, we may assume that no proper subsemigroup of $\langle u\rangle$ maps onto $\langle y\rangle$ via $c$. Note that $\langle u\rangle\notin (\langle y\rangle)$ because $\langle y\rangle$ is free on one generated in the pseudovariety it generates and $c$ does not split.
Then $g=a^{-1}c^{-1}$ is contained in $(f)$ by closure under codomain division.  We claim that $\langle x\rangle$ is not in $(g)\mathfrak q(\langle y\rangle)$.  This follows from~\cite[Proposition~2.4.22]{qtheor}.  Indeed, any subsemigroup $T$ of $\langle u\rangle$ in the pseudovariety generated by $\langle y\rangle$ must be proper and hence map by $c$ into a proper subsemigroup $U$ of $\langle y\rangle$.  Then the image under $a$ of $U$ is proper and so we get something in a proper subpseudovariety of $(\langle x\rangle)$.
\end{proof}

\begin{theorem}
The atoms of $\PVRM$ are $(\id_a)$ with $a=\2^r,\2^l,\{(0,1),\cdot\}$, i.e., with $a$ is a projective atom.
\end{theorem}
\begin{proof}
Theorem~\ref{thm:1.24} proves that these are atoms. Lemma~\ref{lem:1.23} proves that all atoms are of the form $(f)$ where $f\colon S\to T$ with $(S)$ an atom of $\mathbf{PV}$.  If $(S)$ is not one of the pseudovarieties of right zero semigroups, left zero semigroups or semilattices, then $S$ either is a null semigroup or an elementary abelian $p$-group.  But then there is a division $d\colon C\to S$ with $C$ a non-trivial cyclic semigroup and replacing $f$ its divisor $df$, one may assume that $S$ is cyclic and so Lemma~\ref{nocyclicfront} implies that $(f)$ is not an atom.

It remains to show that if $(S)$ generates one of the pseudovarieties of right zero semigroups, left zero semigroups or semilattices, then $(f)=(1_a)$ with $a$ as in the theorem statement.   In this case, $a$ is a subsemigroup of $S$, so replacing $f$ by a restriction, we may assume that $S=a$, that is, $f\colon a\to T$ with $a$ one of the projective semigroups $\mathbf 2^r$, $\mathbf 2^l$  or $\{0,1\}$.

 Diagram $f$ as:
\begin{equation*}
\begin{tikzpicture}
\matrix (m) [matrix of math nodes, row sep=4em,column sep=2em, text height=1.5ex, text depth=0.25ex,ampersand replacement=\&]
{ \#f \& T' \&  T \\
 a \& \&\\ };
\path[>=latex,->>] (m-1-1.east)  edge node[above] {$\beta$}   (m-1-2.west); \path[>=latex,right hook->] (m-1-2.east)  edge    (m-1-3.west);
\path[>=latex,->>] (m-1-1.south)  edge node[left] {$\alpha$}  (m-2-1.north);\path[>=latex,->] (m-2-1.north east)  edge node[below] {$f$}  (m-1-3.south west);
\end{tikzpicture}
\end{equation*}
and as usual, without loss of generality, we can assume
\begin{equation*}
\begin{tikzpicture}
\matrix (m) [matrix of math nodes, row sep=4em,column sep=4em, text height=1.5ex, text depth=0.25ex,ampersand replacement=\&]
{ \#f \& T'  \\
 a \& \\ };
\path[>=latex,->>] (m-1-1.east)  edge node[above] {$\beta$}   (m-1-2.west);
\path[>=latex,->>] (m-1-1.south)  edge node[left] {$\alpha$}  (m-2-1.north);\path[>=latex,->] (m-2-1.north east)  edge node[below] {$f$}  (m-1-2.south west);
\end{tikzpicture}
\end{equation*}
by closure of pseudovarieties of relational morphisms under corestriction.  Then since $a$ is projective, $a$ is a subsemigroup of $\#f$ in such a way that $\alpha|_a$ is an isomorphism  and so $(f)$ contains a homomorphism $a\xrightarrow{\beta}T'$.
 Since $|a|=2$, either $a\cong a\beta$ or $|a\beta|=1$. In the first case $(\id_a)\subseteq (f)$, so $(\id_a)= (f)$ if $(f)$ is an atom. In the second case the collapsing map $c_a\colon a\to \{1\}$ belongs to $(f)$ so again $\id_a\subseteq f$ by \cite[Pages 120--122]{qtheor}) and we are done.
\end{proof}

\begin{corollary}
\begin{itemize}
\item[(a)] The atoms of $\PVRM^-$ are the atoms of $\PVRM$, that is, $\{(\id_a)\mid a=\2^l, 2^r, (\{0,1\}, \cdot )\}$ \up(so $a$ is a projective atom\up).
\item[(b)] The atoms of $\GMC$ \up($\GMC^-$\up) are $(\id_a)\q$, where $a$ is a projective atom.
\end{itemize}
\end{corollary}
\begin{proof}
\begin{itemize}
\item[(a)] Use Fact~\ref{fact:1.3}.
\item[(b)] Use Theorem~\ref{prop:mqadjunction}.
\end{itemize}
\end{proof}
\begin{figure}
\begin{tabular}{|c|p{2.3cm}|p{3.1cm}|p{2.7cm}|p{2.5cm}|}
\hline
& Subset of $\CC$ & Its atoms & Subset $\Cnt(\PV)$ & Its atoms\\
\hline
\hline
1 & $\PVRM$ & $\id_a, a=\2^l, 2^r, (\{0,1\}, \cdot )$  & $\GMC$ & $(\id_a)\q$ \\
\hline
2 & $\PVRM^-$ & same as above & $\GMC^-$ & same as above\\
\hline
\end{tabular}
\caption{Tabulation of results continued.}
\label{tab:results2}
\end{figure}

Knowing about which atoms lift or are very small is related to the atoms of $\GMC^+$ in the following way.

\begin{proposition}
\label{nn:131}
If $a$ one of the atoms $\mathbf 2^r$, $\mathbf 2^l$, $N_2$, $\{0,1\}$ or $\mathbb Z_p$ with $p$ prime that lifts and $(a)$ is very small then $\V\mapsto \V\vee (a)$ is an atom of $\GMC^+$.
\end{proposition}
\begin{proof}
We must show that if $\alpha\in\GMC^+$ satisfies $\id_\PV<\alpha\leq \id_\PV\vee(a)$ then $\alpha=\id_\PV\vee(a)$.

Choose a finite semigroup $S$ so $a\notin (S)$ and $(S)<(S)\alpha\leq(S)\vee(a)$. Then since $(a)$ is very small $(S)\alpha=(S)\vee(a)$.

Choose $\R\in\PVRM^+$ so $\R\q=\alpha$. Then there exists a relational morphism $f\in\R$ diagrammed as
\begin{equation*}
\begin{tikzpicture}
\matrix (m) [matrix of math nodes, row sep=4em,column sep=4em, text height=1.5ex, text depth=0.25ex,ampersand replacement=\&]
{ \#f \& S^n  \\
 S\times a \& \\ };
\path[>=latex,->] (m-1-1.east)  edge node[above] {$\beta_1$}   (m-1-2.west);
\path[>=latex,->>] (m-1-1.south)  edge node[left] {$\alpha_1$}  (m-2-1.north);\path[>=latex,->] (m-2-1.north east)  edge node[below] {$f$}  (m-1-2.south west);
\end{tikzpicture}
\end{equation*}

Let $\varphi_2\colon S\times a\twoheadrightarrow a$ be the projection.  Then  $\alpha_1\phi_2\colon \#f\twoheadrightarrow a$ is surjective and thus, since $a$ lifts, $a\leq\#f$. Now $a$ is congruence-free, i.e., has no non-trivial proper quotients.
Thus $\beta_1$ restricted to $a$ has trivial image $\{e\}$ since $a\notin(S)$. Now by Tilson's Lemma~\cite[Lemma~2.1.9]{qtheor}, valid for elements of $\PVRM^+$, and closure of pseudovarieties of relational morphism under range restriction, $\beta_1$ belongs to $\R$ and hence so does its divisor
the collapsing morphism $a\rightarrow \{e\}$.  Thus $\widetilde{(a)}\leq\R$ (see \cite[Pages~120--121]{qtheor}), which implies $(a)\leq (\W)\alpha$ for all $\W\in\PV$. Therefore, $\alpha=\id_\PV\vee(a)$.
\end{proof}

\begin{theorem}
\begin{itemize}
\label{thm:132}
\item[a)] The atoms $2^l,2^r,\{0,1\},\Z_p$ lift, but $N_2$ does not lift.
\item[b)] $N_2$ and $\{0,1\}$ are very small, but $2^l,2^r,\Z_p$ are not very small. Hence $\V\rightarrow\V\vee(\{0,1\})$ is an atom of $\GMC^+$.
\end{itemize}
\end{theorem}
\begin{proof}
We first prove (a). It is easy to show that $2^l,2^r,\{0,1\}$ are projective and hence lift~\cite[Lemma~4.1.39]{qtheor}. The group $\Z_p$ lifts because if $\varphi\colon  S\rightarrow \Z_p$ is a surjective homomorphism, then  there exists a subgroup $G\leq S$ mapping onto $\Z_p$.  But then $p$ divides $|G|$ and so, by Cauchy's Theorem, $\Z_p\leq G$.  Thus $\Z_p$ lifts (but it is not projective as the canonical map $\Z_{p^2}\to \Z_p$ does not split).

The homomorphism $\varphi\colon \langle y\mid y^2=y^4\rangle\rightarrow N_2$, $y\mapsto n$, $y^2,y^3\mapsto 0$, shows $N_2$ does not lift.

Now we turn to (b).  It is proved in~\cite[Theorem~2.4]{SLsmall} that the pseudovariety of semilattices is very small.  We now prove that $(N_2)$ is very small. It suffices to show if $S$ is completely regular (since $\mathbf{CR}=\Excl(N_2)$; see \cite[Table~7.2, Page~469]{qtheor}) and
$(S)\leq \W < (S)\vee(N_2)=(S\times N_2)$
then $\W\subseteq(S)$. If $N_2\in\W$, then $\W=(S)\vee(N_2)$, hence a$N_2\notin\W$. Thus $\W\subseteq \CR$.
Well, $T$, a member of  $\W$ contained in $\CR$, implies $T$ divides  $S_1\times N_2^m$ with $S_1\in(S)$ (and hence $S_1\in \CR$). Let $U^{\omega}$ be the idempotent power of a semigroup $U$ (viewed as an of the power semigroup $P(U)$).  Since $S$ and $T$ are completely regular, we have  $T=T^\omega$ divides $(S_1\times N_2^m)^{\omega}=  S_1\times(0)\cong S_1$. Thus $T\in (S)$ and we are done.


 Next we show that $2^l,2^r,\Z_p$ are not very small. The idea of this and the following proofs is that if $N$ is nilpotent (i.e., there exists $k$ such that $N^k=0$) then, for any finite semigroup $S$, $(N\times S)/(0\times S)$ is also nilpotent, and has a surjective morphism onto $N$ induced by $(n,s)\mapsto n$. Thus $(N\times S)=(N)\vee(S)$ ``can grow'' larger nilpotents (even if $S\in \CR$ \cite{ABR}).

First  we show $2^r$ is not very small. It is easy to see that if $S=\langle x\mid x^3=0\rangle$ then
\[ (S)<\left(\frac{S\times 2^r}{0\times 2^r}\right)<(S\times 2^r)=(S)\vee(2^r), \]
where the first inequality is strict because $(S\times 2^r)/(0\times 2^r)$ is not commutative as $(x,a),(x,b)$ do not commute.
The second inequality is strict because $(S\times 2^r)/(0\times 2^r)$ is nilpotent. The dual argument shows that $2^l$ is not very small.

Now we prove that  $\Z_p$ is not very small.
This should be considered joint work with M.~Sapir. Let $S$  be the free semigroup on $a,b$ in the variety defined by the identities $x_1x_2x_3x_4=0, x^2y=xy^2$.  Routine computations shows that $|S|=13$ and $S=\{a,b,a^2,ab,ba,b^2,a^2b=ab^2, b^a=ba^2,a^3,aba,bab,b^3,0\}$.
Now consider $G=\Z_p=\langle g\rangle$.
\[ (S)<\left(\frac{S\times G}{0\times G}\right)<(S\times G)=(S)\vee(G) \]
The center term is nilpotent so the second inequality follows. The center term satisfies $x_1x_2x_3x_4=0$, but not $x^2y=xy^2$ since it does not hold in $G$. In detail,
let us substitute $(a,g)$ for $x$ and $(b,g^2)$ for $y$, then
\begin{align*}
(a,g)(a,g)(b,g^2)&=(a^2b=ab^2,g^4)\\
(a,g)(b,g^2)(b,g^2)&=(ab^2=a^2b,g^5)
\end{align*}
 and $g^4\neq g^5$ in $\Z_p$ any $p$. In fact the elements $\bar{a}=(a,g)$, $\bar{b}=(b,g^2)$ in $(S\times G)/(0\times G)$  freely generate a relatively free semigroup in the variety $x_1x_2x_3x_4=0$.  This variety is clearly generated by its free object on two generators and so
\[ \left(\frac{S\times G}{0\times G}\right) = \ldbrack x_1x_2x_3x_4=0\rdbrack\]
Thus $\Z_p$ is not very small. This finishes the proof of b) and hence of Theorem~\ref{thm:132}.
\end{proof}

We note some further open questions regarding the atoms of $\GMC^+$.
\begin{itemize}
\item[(a)] One should check that none of the $\V\rightarrow \V\vee(a)$ are atoms of $\GMC^+$ with $(a)$ an atom of $\PV$ except $a = \{0,1\}$.
\item[(b)] Conjecture: $\Atoms(\GMC^+)=\V$ to $\V \vee (\{0,1\}, \cdot )$.
\end{itemize}

Using the same idea we construct some $\smi$ members of $\PV$ which are not $\mi$, a question posed in~\cite[page~471]{qtheor}.
The following is an extension of joint work with M.~Sapir which considered the two variable case.  Throughout, we use boldface letters (typically, $\bw,\bu,\bv$, sometimes with subscripts) to denote words, and standard lower case letters (typically, $x,y, z$, sometimes with subscripts) to denote letters appearing in words.  The symbol $\equiv$ is used to denote equality between words.  So, $\bw\equiv xyx$ denotes the fact that the word $\bw$ is $xyx$, while $\bw=xyx$ denotes a formal equality that may not hold in the variety of all semigroups (such as if $\bw\equiv xy$ for example).  We use $\con(\bw)$ to denote the \emph{content} of $\bw$: the alphabet of letters appearing in $\bw$.

\begin{proposition}
\label{prop:133}
Consider words $\bw_1\neq \bw_2$, with $\con(\bw_1)=\con(\bw_2)=\{x_1,\dots,x_k\}$ and $|\bw_1|=|\bw_2|=n\geq k>1$. Then the pseudovariety
$\ldbrack \bw_1=\bw_2 \rdbrack$ is $\smi$ but not $\mi$ and has as unique cover
$$\ldbrack \bw_1=\bw_2 \rdbrack \vee \ldbrack T_{\bw_1,\bw_2},  x_1\cdots x_{n+1}=0\rdbrack$$
where $T_{\bw_1,\bw_2}$ consists of all equations $\theta(\bw_1)=\theta(\bw_2)$ for which
$\theta:\{x_1,\dots,x_k\}\to\{x_1,\dots,x_k\}$ has $|\theta(\{x_1,\dots,x_k\})|<k$.
\end{proposition}

An immediate corollary is the following result, which appears in~\cite{qtheor}.

\begin{corollary}
The pseudovariety $\Com=\ldbrack xy=yx\rdbrack$ is $\smi$ but not $\mi$, with unique cover $\Com\vee\ldbrack x_1x_2x_3=0\rdbrack$.
\end{corollary}

\begin{proof}[Proof of Proposition~\ref{prop:133}]
Let $N_{n+1}$ denote the free semigroup on $k$ generators in the variety defined by $x_1\cdots x_{n+1}=0$.  The elements of $N_{n+1}$ are $0$, along with each word in the alphabet $\{x_1,\dots,x_k\}$ of length at most $n$.  Let $N_{n+1}^\flat$ denote the quotient of $N_{n+1}$ by the fully invariant congruence $\rho$ corresponding to the equations in $T_{\bw_1,\bw_2}$.  Note that if $\bu=\bv$ is an equation in $T_{\bw_1,\bw_2}$, and $\theta$ is any substitution, then $\con(\theta(\bu))=\con(\theta(\bv))$ and either $|\con(\theta(\bu))|<k$ or $|\theta(\bu)|>n$.  Hence (as $x_1\cdots x_{n+1}=0$ already holds) the only nontrivial relations in $\rho$ are those corresponding to the transitive closure of the equalities in $T_{\bw_1,\bw_2}$.  We now observe that $N_{n+1}^\flat$ generates the pseudovariety $\ldbrack T_{\bw_1,\bw_2}, x_1\cdots x_{n+1}=0\rdbrack$, which therefore is compact.

To see this, consider an identity $\bu=\bv$ failing in the variety defined by $T_{\bw_1,\bw_2}\cup\{x_1\cdots x_{n+1}=0\}$.  If $\bu,\bv$ are two words of different length, they can be distinguished in the free object on $\{x_1\}$ by sending all letters to $x_1$.  If $\bu,\bv$ have the same length strictly less than $n$ then find a position in which the letter appearing in $\bu$ is distinct from that appearing in $\bv$; say, $x$ appears at the $i$th position of $\bu$ and $y$ appears at the $i$th position of $\bv$.  Take any substitution from $\con(\bu\bv)$ into $\{x_1,\dots,x_k\}$ that separates $x$ from $y$.  Then this witnesses failure of $\bu=\bv$ on $N_{n+1}^\flat$ because distinct products in $\{x_1,\dots,x_k\}$ of length less than $n$ are distinct in $N_{n+1}^\flat$.  The remaining case is where $\bu,\bv$ both have length $n$.  If one of $\bu$ or $\bv$ has at least $k$ variables (say, $\bu$), then again select a position where $\bu$ and $\bv$ differ, and select an assignment $\theta$ mapping $\con(\bu)$ \emph{onto} $\{x_1,\dots,x_k\}$ and which separates the letters in this position.
Then $\theta(\bu)$ involves all $k$ letters and has length $n$, and hence is distinct in $N_{n+1}^\flat$ from every other word in $\{x_1,\dots,x_k\}^*$, and in particular, to $\theta(\bv)$.  So finally, assume that $\bu$ and $\bv$ have length $n$ and both involve fewer than $k$ letters.  But then $N_{n+1}^\flat$ fails $\bu=\bv$ because it is free, on $k$ free generators.

Now let $S$ denote the quotient of $N_{n+1}$ by the fully invariant congruence generated by $\bw_1=\bw_2$.  Because $T_{\bw_1,\bw_2}$ already accounted for all consequences of $\bw_1=\bw_2$ in fewer than $k$ variables (and there were none of length less than $n$), the semigroup $S$ differs from $N_{n+1}$ only amongst those words of length $n$ and in exactly $k$ variables.
Of course, $S\in\ldbrack \bw_1=\bw_2 \rdbrack$.
Now assume $\V\in\PV$, with $\V>\ldbrack \bw_1=\bw_2 \rdbrack$.   We show that $N_{n+1}^\flat \in\V$.

Now, there must be $T\in \V$ not satisfying $\bw_1=\bw_2$.  So there exists $t_1,\dots,t_k\in T$ with $\bw_1(t_1,\dots,t_k)\neq \bw_2(t_1,\dots,t_k)$.
Consider $\widetilde{S}=(S\times T^{k!})/(0\times T^{k!})$ which is a member of $\V$ because $S$ and $T$ are.  Fix an enumeration $\pi_1,\dots,\pi_{k!}$ of the permutations of $\{1,\dots,k\}$ and consider the subsemigroup $F$ of $\widetilde{S}$ generated by the elements $\bar{a_1},\dots,\bar{a_k}$ defined as follows. The value of $\bar{a_i}$ in the $S$ coordinate is $x_i$.  At the $j$th $T$ coordinate, $\bar{a_i}$ is $t_{i\pi_j}$.

We show that $N_{n+1}^\flat$ is a quotient of $F$.  Now, $F$  is $k$-generated and $n+1$-nilpotent, so it is a homomorphic image of $N_{n+1}$ under some homomorphism $\eta$ mapping the free generators by $x_i\mapsto \bar{a_i}$.  We need to show that $\ker(\eta)\subseteq \rho$ (the fully invariant congruence on $N_{n+1}$ yielding $N_{n+1}^\flat$).  The projection from $S\times T^{k!}$ induces a surjective homomorphism $\widetilde{S}\to S$ whose restriction to $F$ is surjective, and moreover maps $\bar{a_i}\mapsto x_i$ for each $i$.  Thus if $\bu$ and $\bv$ are words in $x_1,\dots,x_k$ that represent distinct elements of $S$, then $\bu(\bar{a_1},\dots,\bar{a_k})\neq \bv(\bar{a_1},\dots,\bar{a_k})$ in $F$ also.  Because $S$ differs from $N_{n+1}^\flat$ only on words of length $n$ involving all $k$ letters, to show $\ker(\eta)\subseteq \rho$ it suffices to show that distinct words $\bu$ and $\bv$  of length $n$ and with $\con(\bu)=\con(\bv)=\{x_1,\dots,x_k\}$ have $\bu(\bar{a_1},\dots,\bar{a_k})\neq \bv(\bar{a_1},\dots,\bar{a_k})$.  This is true already if $\bu=\bv$ fails on $S$.  So assume that $\bu=\bv$ holds in $\ldbrack \bw_1=\bw_2 \rdbrack$.  In this case there is a permutation $\pi$ of $\{1,\dots,k\}$ with $\bu(x_1,\dots,x_k)=\bw_1(x_{1\pi},\dots,x_{k\pi})$ and $\bv(x_1,\dots,x_k)=\bw_2(x_{1\pi},\dots,x_{k\pi})$ or vice versa.  Then  $\bu(\bar{a_1},\dots,\bar{a_k})$ differs from $\bv(\bar{a_1},\dots,\bar{a_k})$ on the coordinate corresponding to $\pi^{-1}$.  Thus $\ker(\eta)\subseteq \rho$, and $N_{n+1}^\flat$ is a homomorphic image of $F$.  Hence $N_{n+1}^\flat\in \V$ as claimed.

This proves $\ldbrack \bw_1=\bw_2 \rdbrack$ is $\smi$. It cannot be $\mi$, since no $\mi$ satisfies an identity since each $\mi$ pseudovariety must contain $\G$ or $\Ap$ and these  satisfy no identities. This proves Proposition~\ref{prop:133}.
\end{proof}

The following proposition is well known.

\begin{proposition}\label{prop:notlocallyfinite}
Let $E$ be a set of identities over an alphabet $A$.  Then the \textbf{pseudovariety} $\ldbrack E\rdbrack$ is locally finite if and only if there are no infinite, finitely generated, residually finite semigroups in the \up(Birkhoff\up) \textbf{variety} $\ldbrack E\rdbrack$.
\end{proposition}
\begin{proof}
Suppose first that $\ldbrack E\rdbrack$ contains an infinite, finitely generated, residually finite semigroup $S$.  Let $A$ be a finite generating set for $S$.  Then $S$ has finite quotients of arbitrarily large size, all of which belong to the pseudovariety $\ldbrack E\rdbrack$.  Thus $\ldbrack E\rdbrack$ cannot be locally finite.  Conversely if $\ldbrack E\rdbrack$ is not locally finite, then there is a finite alphabet $A$ such that the free pro-$\ldbrack E\rdbrack$ semigroup $\widehat F$ on $A$ is infinite.  The abstract subsemigroup $S$ of $\widehat F$ generated by $A$ is then an infinite $A$-generated residually finite semigroup in the variety $\ldbrack E\rdbrack$.
\end{proof}

Recall that an identity $\bw_1=\bw_2$ over an alphabet $A$ is \emph{balanced} if the number of occurrences in each letter in $A$ is the same in both $\bw_1$ and $\bw_2$. In this case, $(\mathbb N,+)$ satisfies the identity $\bw_1=\bw_2$ and since $\mathbb N$ is residually finite, it follows from the above proposition that $\ldbrack \bw_1=\bw_2\rdbrack$ is not locally finite and hence not compact.  Thus we have the following proposition.

\begin{proposition}\label{prop:balancedcase}
If $\bw_1=\bw_2$ is a balanced identity satisfying the properties of Proposition~\ref{prop:133}, then $\ldbrack \bw_1=\bw_2\rdbrack$ is a non-locally finite $\smi$, and hence, in particular, is not compact.
\end{proposition}
Recall that a word $\bw$ is \emph{avoidable} if there is a finite alphabet $A$ and an infinite factorial subset of $A^*$ avoiding   $\bw\theta$ for every $\theta:\con(\bw)^*\to A^*$; equivalently there is a right infinite word $\bx\in A^\mathbb{N}$ avoiding $\bw\theta$ for every $\theta:\con(\bw)^*\to A^*$.  The word $\bw$ is \emph{unavoidable} if it is not avoidable.  Recall the \emph{Zimin words}, which are defined inductively by $\bz_1=x_1$, $\bz_{n+1}=\bz_nx_{n+1}\bz_n$.  It is known that a word $\bw$ is unavoidable if and only if there is a substitution $\theta$ with $\theta(\bw)\leq \bz_n$ for some $n$; see Bean, Ehrenfeucht, McNulty \cite{BEM}, Zimin \cite{Zim} or Lothaire \cite{lothaire}.
\begin{proposition}\label{prop:unbalancedcase}
Suppose that $\bw_1,\bw_2\in \{x_1,\dots,x_k\}^+$ are both avoidable words.  Then the pseudovariety $\ldbrack \bw_1=\bw_2\rdbrack$ is not locally finite and hence not compact.
\end{proposition}
\begin{proof}
There is a a finite alphabet $A$, and an infinite sequence $\bu$ on $A$ which avoids images of both $\bw_1$ and $\bw_2$ (see \cite[Corollary 3.2.9]{lothaire} for example).  Let $I(\bu)$ be the ideal of $A^+$ consisting of the non-factors of $\bu$.  Then $S=A^+/I(\bu)$ is an infinite semigroup satisfying $\bw_1=\bw_2=0$ since any evaluation of $\bw_1$ and $\bw_2$ in $S$ will result in $0$ because $\bw_1,\bw_2$ are avoided by $\bu$. It is residually finite because if $I_n$ is the ideal of words in $A^+$ of length greater than or equal to $m$, then the projections $S\to A^+/(I(\bu)\cup I_n)$ separate points.  Thus $\ldbrack \bw_1=\bw_2\rdbrack$ is not locally finite by Proposition~\ref{prop:notlocallyfinite}.
\end{proof}
To achieve a compact $\smi$ it follows from Proposition~\ref{prop:balancedcase} that we need $n>k$ in Proposition~\ref{prop:133}.  The smallest choice is then $n=3$ and $k=2$, for which there are four possible cases: $x^2y=yx^2$, $x^2y=yxy$, $xy^2=xyx$ and $xyx=yxy$.  The first of these involves  avoidable words only, hence by Proposition~\ref{prop:unbalancedcase} does not define a compact pseudovariety.  In the next section we will show that the remaining three pseudovarieties $\ldbrack x^2y=yxy\rdbrack$, $\ldbrack xy^2=xyx\rdbrack$ and $\ldbrack xyx=yxy\rdbrack$ are indeed compact.  We then use these to generate an infinite family of compact $\smi$ examples.

%

\section{Compact $\smi$ pseudovarieties}\label{sec:compact}
Following Proposition~\ref{prop:133}, the pseudovarieties $\ldbrack xyx=xyy\rdbrack$, $\ldbrack xyx=yyx\rdbrack$ and $\ldbrack xyx=yxy\rdbrack$ are $\smi$.  We now show that each is compact, thus answering a central part of Problem 36 in \cite{qtheor}. 
The main difficulties are in finding equational deductions for various consequences of the given axiom.  While this was done by hand, the authors also used Prover9 for a separate verification.  Recall that we use $\equiv$ between words to denote the fact that the words are identical.  So $xy\not\equiv yx$ as the two sides are distinct, while $\bw\equiv xy$ would denote the fact that the word $\bw$  is the actual string $xy$ (where $x,y$ are letters).  In the context of an equational deduction, we place an equation number over the top of an equality sign to indicate which law is being applied.  We use bracketing mostly to specify the precise subword to which the application is being applied, while an underline indicates the subword obtained during the previous deduction.

\subsection{$\ldbrack xyx=xyy\rdbrack$ and $\ldbrack xyx=yyx\rdbrack$}
We consider the variety generated by
\begin{equation}
xyx= xyy.\label{eq:main}
\end{equation}
with the case $xyx=yyx$ following by symmetry.
\begin{lemma}
The following are consequences of equation \eqref{eq:main}\up:
\begin{align}
&x^4= x^5\label{eq:period}\\
& xyz^2= xyz^3= xyz^4\label{eq:flex}\\
& xy^3= xy^4\label{eq:2flex}\\
& x^2y^2= x^2y^4= x^2y^3.\label{eq:3flex}
\end{align}
\end{lemma}
\begin{proof}
Proof of \eqref{eq:period}. By assigning $x\mapsto x$ and $y\mapsto x^2$ we obtain $x^4\equiv x(x^2)x\rapprox{eq:main}x(x^2)^2\equiv x^5$.

Proof of \eqref{eq:flex}. We first show that $xyz^2= xyz^4$.  We have
$[xyz^2]\rapprox{eq:main} \underline{[xyzx]y}\rapprox{eq:main} \underline{xyz[yz}y]\rapprox{eq:main}
x[yz\underline{y]z^2}\rapprox{eq:main}
x\underline{yz^2}z^2\equiv xyz^4$.
This then gives $xyz^3\equiv xyz^2z= xyz^4z\rapprox{eq:period}xyz^4$.

\noindent Proof of (\ref{eq:2flex},\ref{eq:3flex}).  These are consequences of \eqref{eq:flex}: $xy^3\equiv  xy y^2\rapprox{eq:flex}xyy^3\equiv xy^4$, while $x^2y^2\rapprox{eq:flex}x^2y^4$.  And $x^2y^2\equiv xxy^2\rapprox{eq:flex}xxy^4\equiv x^2y^4\rapprox{eq:period}x^2y^4y$.  Now applying $x^2y^2= x^2y^4$ from right to left we obtain $x^2y^2= x^2y^3$.
\end{proof}

\begin{lemma}\label{lem:normalform}
If $\bw$ is a word in letters $x_1,\dots,x_n$, with each letter appearing and with leftmost appearances of the letters in the given order.  Then $\bw$ is equivalent under \eqref{eq:main} to the word
\[
x_{1}^{i_1}x_{2}^{i_2}\cdots x_{n}^{i_n}
\]
for some $i_1\in\{1,2,3,4\}$, $i_2\in\{1,2,4\}$ and $i_j\in\{1,4\}$ for $j>2$ and such that if $i_1>1$ then $i_2\in\{1,4\}$.
\end{lemma}
\begin{proof}
We first reduce to an intermediate form where the $i_j$ may be any number between $1$ and $4$.
Let $i$ be smallest such that $\bw$ has a subword of the form $x_i\bu x_i$, with no occurrences of $x_i$ in $\bu$: if there are no such $i$ then $\bw$ is already in the intermediate form just described.  Otherwise though, let $\bw_i$ denote the prefix of $\bw$ up to but not including the left-most occurrence of $x_i$.  Apply \eqref{eq:main} to replace $x_i\bu x_i$ by $x_i\bu\bu$.  Note that the number of occurrences of $x_i$ goes down under this application of~\eqref{eq:main}, but the prefix $\bw_i$ is unchanged.  Thus we may repeat this for $x_i$ until eventually arriving at $\bw= \bw_ix_i^{j_i}\bv$, where $\bv$ contains no occurrences of $x_i$, and $j_i> 0$.  Now search for the next value $i$, as the smallest number for which this new word there is a subword of the form $x_i\bu x_i$, with no occurrences of $x_i$ in $\bu$.  Repeat until there are no more such $i$.  Denote the resulting intermediate word as $\bw'$.

Now use equation \eqref{eq:period} to reduce any powers of letters in $\bw'$ to at most $4$.  Now if $i\geq 3$ and $x_i$ is nonlinear in $\bw'$, then equation \eqref{eq:flex} can be used to replace this power by $4$.  Similarly if the power of $x_2$ is $3$, then equation~\eqref{eq:2flex} shows that it can be raised to $4$.  If the power of $x_2$ is $2$ and the power of $x_1$ is not $1$, then equation~\eqref{eq:3flex} shows that $x_2$ may be raised to the power $4$.  This completes the proof.
\end{proof}

We now give a finite generator for the variety defined by $xyx= xyy$.  
This generator was found using the aid of Mace4, and while a full justification for the validity of the example is given in the proof of Theorem~\ref{thm:xyx_xyy} below, we first briefly describe the technique for discovery. 
As an initial step, we observed by syntactic arguments that whenever $\bu= \bv$ is an equation between distinct normal forms, then by identification of variables, there are distinct normal forms $\bu'$ and $\bv'$ in at most $3$ variables and such that $\bu= \bv\vdash \bu'= \bv'$.  
This is a consequence of Lemma~\ref{lem:normalform}: this already shows that $\ldbrack xyx=xyy\rdbrack$ is compact, as it shows that the three-generated relatively free algebra, which is finite, generates the pseudovariety.  
To find a smaller generator, it is then only necessary to find small models of $xyx= xyy$ that fail such identities.  
These can be found, one by one, using Mace4.  To get the single small generator ${\bf B}$ we fixed the assumptions $x(yz)= (xy)z, x(yx)= x(yy)$, and searched for counterexamples for the various cases encountered in the proof of Theorem~\ref{thm:xyx_xyy} below.  The most fruitful approach was to first find a counterexample to the single case $x^3y^4= x^4y^4$, which yields the subsemigroup on $\{0,1,\dots,7\}$.  
This is then added to the assumptions and a search for a counterexample to $x^4yz^4= x^4y^2z^4$ is initiated.  
This produces semigroup ${\bf B}$. The two searches take only a few seconds.
\begin{table}
\[\begin{tabular}{r|rrrrrrrrrrr}
$*$ & 0 & 1 & 2 & 3 & 4 & 5 & 6 & 7 & 8 & 9 & 10\\
\hline
    0 & 2 & 3 & 4 & 5 & 6 & 7 & 6 & 6 & 5 & 5 & 5 \\
    1 & 1 & 1 & 1 & 1 & 1 & 1 & 1 & 1 & 1 & 1 & 1 \\
    2 & 4 & 5 & 6 & 7 & 6 & 6 & 6 & 6 & 7 & 7 & 7 \\
    3 & 3 & 3 & 3 & 3 & 3 & 3 & 3 & 3 & 3 & 3 & 3 \\
    4 & 6 & 7 & 6 & 6 & 6 & 6 & 6 & 6 & 6 & 6 & 6 \\
    5 & 5 & 5 & 5 & 5 & 5 & 5 & 5 & 5 & 5 & 5 & 5 \\
    6 & 6 & 6 & 6 & 6 & 6 & 6 & 6 & 6 & 6 & 6 & 6 \\
    7 & 7 & 7 & 7 & 7 & 7 & 7 & 7 & 7 & 7 & 7 & 7 \\
    8 & 9 & 3 & 10 & 3 & 10 & 10 & 10 & 10 & 8 & 9 & 10 \\
    9 & 10 & 3 & 10 & 10 & 10 & 10 & 10 & 10 & 10 & 10 & 10 \\
    10 & 10 & 10 & 10 & 10 & 10 & 10 & 10 & 10 & 10 & 10 & 10
\end{tabular}
\]
\caption{The semigroup $\mathbf{B}$, a generator for $\ldbrack xyx=xyy \rdbrack$.}\label{tab:B}
\end{table}
\begin{theorem}\label{thm:xyx_xyy}
The variety defined by $xyx= xyy$ is generated by the  semigroup ${\bf B}$ of Table~\ref{tab:B}.\\
\end{theorem}
\begin{proof}
It is routinely verified that ${\bf B}$ is a semigroup satisfying $xyx= xyy$.  Thus it will suffice to show that if $\bu= \bv$ is an equation that does not follow from $xyx= xyy$ then $\bu= \bv$ fails on ${\bf B}$.  
So let $\bu= \bv$ be an identity that does not follow from $xyx= xyy$.  
By  Lemma~\ref{lem:normalform}, there is no loss of generality to assume that $\bu$ and $\bv$ are in normal form.

If $\bu$ or $\bv$ have distinct alphabets, or if the order of occurrence of the letters is not identical, then $\bu= \bv$ will fail on the subsemigroup $\{8,3,10\}$ of ${\bf B}$, as this semigroup is isomorphic to the monoid obtained from adjoining an identity element to $\2^l$ (where $8$ plays the role of the identity element).

Thus we may assume that there is a number $n>0$ such that $\bu\equiv x_1^{\alpha_1}\cdots x_n^{\alpha_n}$ and $\bv\equiv x_1^{\beta_1}\cdots x_n^{\beta_n}$ where $\alpha_1,\dots,\alpha_n$ and $\beta_1,\dots,\beta_n$, with the $\alpha_i$ and $\beta_i$ satisfying the constraints on indices in normal forms outlined in Lemma~\ref{lem:normalform}.  As $\bu\neq \bv$ there is some $i$ such that $\alpha_i\neq \beta_i$, and without loss of generality we may assume that $\alpha_i<\beta_i$.  If $\alpha_i=1$ for some $i\leq n$ then consider the evaluation $\theta_1$ into ${\bf B}$ defined by $x_i\mapsto 0$ and
\[
x_j\mapsto\begin{cases}
8&\text{ if }j<i\\
1&\text{ if }j>i.
\end{cases}
\]
Then $\theta_1(\bu)= 8 *0* 1= 3$ (or $0*1=3$ if $i=1$, or $8*0=9$ if $i=n$), while because $\beta_1>1$ we have $\theta_1(\bv)= 8 * 0^{\beta_1} * 1=10$ (or $0^{\beta_1}*1\in\{5,6,7\}$ if $i=1$, or $8 * 0^{\beta_1}=10$ if $i=n$, respectively).  In each case, $\theta_1(\bu)$ and $\theta_1(\bv)$ take different values in ${\bf B}$ as required.

Thus we may assume in remaining cases that if $\alpha_j=1$ if and only if $\beta_j=1$ for each $j=1,\dots,n$.
If $i=1$ and $\alpha_1\in\{2,3\}$ (so that $\beta_1\in\{2,3,4\}\backslash\{\alpha_1\}$), then use the evaluation $\theta_2$ into ${\bf B}$ defined
$x_1\mapsto 0$ and assigning all other letters to $1$.  Then $\theta_2(\bu)=0^{\alpha_1}1$, while $\theta_2(\bv)=0^{\beta_1}1$.  If $\alpha_1=3$ then $\beta_1=4$ and we have $\theta_2(\bu)=0^3*1=4*1=7$ while $\theta_2(\bv)=0^4*1=6*1=6$.  If $\alpha_1=2$, then $\theta_2(\bu)=0^2*1=2*1=5$, while $\theta_2(\bv)\in\{0^3,0^4\}*1=\{4*1,6*1\}=\{6,7\}$.  Thus $\theta_2(\bu)\neq \theta(\bv)$ in ${\bf B}$ as required.

Thus we may assume that $\alpha_1=\beta_1$.  
Looking at the constraints on indices for normal forms, we see that there is only one further way that $\bu$ and $\bv$ can differ: if $\alpha_1=\beta_1=1$ and $\alpha_2=2$ and $\beta_2=4$.  In this case, consider the evaluation $\theta_3$ into ${\bf B}$ defined by $x_1,x_2\mapsto 0$ and $x_j\mapsto 1$ for all $j>2$.  
Then $\theta_3(\bu)=0^3*1=4*1=7$ while $\theta_3(\bv)=0^{1+\beta_2}*1=0^4*1=6$, because $\beta_2\geq 3$.

Thus we have shown that every $\bu,\bv$ with $xyx= xyy\not\vdash \bu= \bv$ we also have ${\bf B}$ fails $\bu= \bv$, which shows that ${\bf B}$ generates the variety defined by $xyx= xyy$.
\end{proof}

\begin{remark}\label{rem:xyxxyy}
The pseudovariety $\ldbrack xyx=xyy \rdbrack$ has precisely two maximal sub-pseudovarieties.  These are defined taking the law $xyx=xyy$ in conjunction with exactly one of the following laws: $x^3y^4=x^4y^4$, and $x^4yz^4=x^4y^4z^4$.
\end{remark}
\begin{proof}
First observe that, when combined with $xyx=xyy$, each of the two equations listed defines a proper subvariety of that defined by $xyx=xyy$, because in each equation the two sides are distinct normal forms.  Consider then an equation $\bu=\bv$ between two distinct normal forms; we must show that one of the two listed equations is a consequence of $\{\bu=\bv,xyx=xyy\}$.  It is useful to note that the equation $x^4y^4z^4=x^4z^4y^4$,  in conjunction with $xyx=xyy$ implies $x^4yz^4=x^4y^4z^4$.  Indeed,  the expression $x^4 (yz^4)^4 y^4$ has normal form $x^4yz^4$ while $x^4 y^4 (yz^4)^4$ has normal form $x^4y^4z^4$.  But a single application of $x^4y^4z^4=x^4z^4y^4$ yields the consequence $x^4 (yz^4)^4 y^4=x^4 y^4 (yz^4)^4$.

If $\bu$ involves a letter $x$ not appearing in $\bv$, then $\bu$ reduces to one of the normal forms $x^4$, $x^4y^4$ or $y^4x^4$, while $\bv$ reduces to $y^4$.  It is not hard to verify that $x^4y^4z^4=x^4z^4y^4$ (hence $x^4yz^4=x^4y^4z^4$) is a consequence of each of the possible resulting laws.  Now assume that $\bu$ and $\bv$ have the same alphabet  $x_1,\dots,x_n$, with the given numbering reflecting the order of first appearance of the letters in $\bu$.

Assume that the order of first appearance of letters in $\bv$ is not the same as in $\bu$.  Let $i$ be the smallest index such that $x_i$ does not make its first appearance first after the first appearance of $x_{i-1}$ (or $i=1$ if $\bv$ starts with a letter other than $x_1$).  Assign all letters $x_1,\dots,x_{i-1}$ the value $x^4$, assign $x_i$ the value $y^4$ and assign all remaining letters the value $z^4$.  Then $\bu$ reduces to $x^4y^4z^4$ (or $y^4z^4$ if $i=1$) while $\bv$ reduces to $x^4z^4y^4$ (or $z^4y^4$ if $i=1$).  This yields $x^4y^4z^4=x^4z^4y^4$ (hence $x^4yz^4=x^4y^4z^4$).

Now we may assume that $\bu\equiv x_1^{i_1}x_2^{i_2}\dots x_n^{i_n}$ and $\bv\equiv x_1^{j_1}x_2^{j_2}\dots x_n^{j_n}$, both normal forms, but with $(i_1,\dots,i_n)\neq (j_1,\dots,j_n)$.  If $n=1$ then we easily obtain law $x^3y^4=x^4y^4$, so assume that $n>1$.  

 If $i_1< j_1$, then by  fixing $x_1$ and assigning all remaining letters the value $y^4$ we obtain $x_1^{i_1}y^4=x_1^{j_1}y^4$.  If $i_1=1$, then replace $x_1$ by $x^2$ to obtain $x^2y^4=x^4y^4$, from which $x^3y^4=x^4y^4$ is a consequence.  If $i_1\in\{2,3\}$, then we may also deduce $x^3y^4=x^4y^4$ directly from $x_1^{i_1}y^4=x_1^{j_1}y^4$.  Assume now that $i_1=j_1$.
 
 If $i_2<j_2$ and $i_1=j_1=1$, then we may obtain $x_1x_2^{i_2}y^4=x_1x_2^{j_2}y^4$. Note that $i_2<j_2$ implies $i_2\in\{1,2\}$.  Replace $x_1$ by $x$, $x_2$ by $x^{3-i_2}$ to obtain $x^3y^4=x^4y^4$.  Now assume that $i_1=j_1>1$, so that $i_2<j_2$ implies $i_2=1$ and $j_2=4$.  Then we obtain $x^4yz^4=x^4y^4z^4$.
 
 Now assume that $i_1=j_1$ and $i_2=j_2$ but $i_k=1$ and $j_k=4$, for some $k\in\{3,\dots,n\}$.  Then the law $x^4yz^4=x^4y^4z^4$ is an easy consequence.
 
 Finally, we note that the subpseudovarieties are distinct: Mace4 finds a model $5$-element model of 
 $\{xyx=xyy,x^3y^4=x^4y^4\}$ failing $x^4yz^4=x^4y^4z^4$, and an $8$-element model of $\{xyx=xyy,x^4yz^4=x^4y^4z^4\}$ failing $x^3y^4=x^4y^4$.
 \end{proof}
\subsection{$\ldbrack xyx=yxy\rdbrack$}
Now we show that following law defines a compact pseudovariety:
\begin{equation}
xyx\feq yxy.\label{eq:main2}
\end{equation}
The `bracketed' center and  the `brackets' can be exchanged.
As consequences the following equalities can be derived.
\begin{lemma}[Periodicity]
\begin{equation}
x^4\feq x^5\label{eq:period2}
\end{equation}
\end{lemma}
\begin{proof}
By assigning $x\mapsto a$ and $y\mapsto a^2$ we obtain $a^4\equiv a(a^2)a\rfeq{eq:main2}(a^2)a(a^2)\equiv a^5$.
\end{proof}

\begin{lemma}[Inside out]
For any $n,m\geq 0$\up:
\begin{align}
xyzx &\feq y^nxyzxy^m\label{eq:insideout}\\
xyzx &\feq z^nxyzxz^m\label{eq:insideout2}
\end{align}
\end{lemma}
\begin{proof}
  First
\begin{equation*}
[abca]\rfeq{eq:main2}\underline{b[cabc]}\rfeq{eq:main2}b\underline{abcab}.
\end{equation*}
Apply this four times to achieve $abca\feq [b^4]abca[b^4]\rfeq{eq:period2}\underline{b^n[b^4}abca\underline{b^4]b^m}\feq b^n\underline{abca}b^m$.  Law \eqref{eq:insideout2} follows by symmetry.
\end{proof}

\begin{lemma}[Outside in]
\begin{equation}
xyzx\feq xyxzx \label{eq:outsidein}
\end{equation}
\end{lemma}
\begin{proof} We have
\[
abca\rfeq{eq:insideout}[b^4]abcab^4\rfeq{eq:period2}\underline{b[b^4}abcab^4]\rfeq{eq:insideout}[bab]ca\rfeq{eq:main2}\underline{aba}ca.
\]
\end{proof}

\begin{lemma}[Bump up bracket powers]
For any $n,m\geq 1$\up:
\begin{equation}
xyzx\feq x^nyzx^m\label{eq:bumpupbracket}
\end{equation}
\end{lemma}
\begin{proof}
$[abca]\rfeq{eq:insideout}[babcab]\rfeq{eq:insideout}a[babcab]\rfeq{eq:insideout}a\underline{abca}$.  The law $abca= abcaa$ follows by symmetry.
\end{proof}
\begin{lemma}[Inside commuting]
\begin{equation}
xyzx\feq xzyx\label{eq:commuting2}
\end{equation}
\end{lemma}
\begin{proof}
First
\begin{align*}
[abca]&\rfeq{eq:bumpupbracket}
a[abca] \rfeq{eq:outsidein}
[a\underline{abaaca}]\\
&\rfeq{eq:main2}
abaac\, a\, abaac\equiv [aba\, aca\, aba] ac\\
&\rfeq{eq:main2}
\underline{aca\, aba\, a[ca}\, ac]\\
&\rfeq{eq:main2}
[aca aba a \underline{aaca]a}\\
&\rfeq{eq:outsidein}\cdots\rfeq{eq:outsidein}
[\underline{acbca}a]\rfeq{eq:bumpupbracket}
acbca.
\end{align*}
Then by symmetry we have $xyx\feq yxy\vdash abca\feq acba$ as required.
\end{proof}

Once the bracketed part is more than one symbol in length, we can independently bump up the powers inside.
\begin{lemma}[Bumping up inner powers]
For $n,m>1$\up:
\begin{equation}
xyzx\feq xy^nz^mx \label{eq:bumpupcenter}
\end{equation}
\end{lemma}
\begin{proof}
Inner part to the outer bracket, iterated insertion of the bracket, then removing bracket.
\[
[abca]\rfeq{eq:insideout}[babcab]\rfeq{eq:outsidein}\ldots\rfeq{eq:outsidein}[bab(b)^{n-1}cab]\rfeq{eq:insideout}ab(b)^{n-1}ca\equiv  ab^{n}ca.
\]
\end{proof}

\begin{lemma}[Inside commuting 2]
For $u,v,w$ either variables or possibly empty\up:
\begin{equation}
xuyvzwx\feq xuzvywx\label{eq:commuting22}
\end{equation}
\end{lemma}
\begin{proof}
It suffices to show that $xuyzwx\feq xuyzwx$ where $u,w$ are possibly empty, as this enables commutativity between any two occurrences of a variable (and $xuyvzwx\feq xuzvywx$ follows).

We have $aubcwa\feq auabcawa$ by \eqref{eq:outsidein} if $u,w$ are nonempty, or by \eqref{eq:bumpupbracket} when one of $u,w$ is empty.  Then $au[abca]wa\rfeq{eq:commuting2}[au\underline{acba}wa]\rfeq{eq:outsidein}aubcwa$, where again \eqref{eq:bumpupbracket} is used in place of \eqref{eq:outsidein} when $u$ or $w$ is empty.
\end{proof}

\begin{lemma}[Leapfrog]
Assume that $u,v,w$ are either variables or empty, with $uvw$ not empty.  Then
\begin{equation}
xyxy\feq xyyx\qquad\text{and}\qquad
xuyvxwy\feq xuvwyx\label{eq:leapfrog}
\end{equation}
\end{lemma}
\begin{proof}
First observe that $[aubva]wb\rfeq{eq:commuting22}\underline{a[buva}wb]\rfeq{eq:commuting22}a\underline{b[uvwa]b}\rfeq{eq:main2}a\underline{uvwa\, b\, uvwa}$, regardless of whether or not $uvw$ is empty.  If $uvw$ is empty, then $[aaba]\rfeq{eq:bumpupcenter}[aabba]\rfeq{eq:bumpupbracket}abba$ as required.  If $uvw$ is nonempty, then
\[
[auvwabuvwa]\rfeq{eq:commuting22}a[auuvvwwba]\rfeq{eq:bumpupcenter}[a\underline{auvwba}]\rfeq{eq:bumpupbracket}auvwba.
\]
\end{proof}

\begin{lemma}[Evert]
For $u,v$ possibly empty:
\begin{equation}
xuyvx\feq yuxvy\label{eq:evert}
\end{equation}
\end{lemma}
\begin{proof}
For $uv$ empty, this is \eqref{eq:main2}.  Without loss of generality, assume that $u$ is nonempty (with $v$ either empty or nonempty).  Then
$xuyvx\rfeq{eq:bumpupcenter}xuyyvx\rfeq{eq:commuting22}xyuvyx\rfeq{eq:main2}yuvyxyuvy$.  Then applying \eqref{eq:commuting22} and \eqref{eq:bumpupcenter} reduces this word to $yuxvy$.
\end{proof}

A word ${\bf w}$ is said to be \emph{connected} if there are letters $x_1,\dots,x_n$ (for $n>1$) such that
\[
\bw\equiv x_1\cdots x_2\cdots x_1\cdots x_3\cdots x_2\cdots x_4\cdots \qquad \cdots x_{n}\cdots x_{n-1}\cdots x_n.
\]
When $n=1$ it is convenient to require that $\bw$ is of the form $x_1\cdots x_1$, and not simply $x_1$.
A connected word $\bw$ whose variables are $x_1,\dots,x_n$ is said to be in \emph{canonical form} if it satisfies the following.
\begin{enumerate}
\item[(i)] If $n=1$, then $\bw\in\{x_1^2,x_1^3,x_1^4\}$.
\item[(ii)] If $n=2$, then $\bw\in\{x_1x_2x_1,x_1x_2^2x_1\}$.
\item[(ii)] If $n>2$ then $\bw\equiv x_1x_2\cdots x_nx_1$.
\end{enumerate}
\begin{lemma}\label{lem:canonical2}
If $\bw$ is a connected word in alphabet $x_1,\dots,x_n$ then there is a word $\bw'$ in canonical form with ${\bf w}\rfeq{eq:main2}\bw'$.
\end{lemma}
\begin{proof}
Let $\bw$ be a connected word in the alphabet $x_1,\dots,x_n$ (all letters appearing).  If $n=1$ the lemma follows immediately from \eqref{eq:period2}.
Now assume $n>1$.  Let $x_i$ be the first letter appearing in $\bw$.  Repeated left-to-right applications of \eqref{eq:leapfrog} will move the final occurrence of $x_i$ further right, eventually resulting in a word $\bw'$ of the form $\bw'\equiv x_i\bu x_i$, where $\bw'$ has the same alphabet as $\bw$.  If $i\neq 1$, then we may write $\bu\equiv x_i\bu_1x_1\bu_2x_i$, where $\bu_1,\bu_2$ are possibly empty.  Then, $\bw\feq \bw'\rfeq{eq:evert}x_1\bu_1x_i\bu_2x_1$.
Then use \eqref{eq:commuting22} to rearrange $\bu_1x_i\bu_2$ into the form $x_1^{i_1}\cdots x_n^{i_n}$, where $i_1\geq 0$ and $i_j\geq 1$ for each $j>1$.  If $n>2$, then we may use \eqref{eq:bumpupcenter} and \eqref{eq:bumpupbracket} to obtain $\bw= x_1x_2\cdots x_nx_1$.   If $n=2$, then we have $\bw= x_1x_2x_1$ or $\bw= x_1x_1^{i_1}x_2^{i_2}x_1$.  If $i_1>0$, then applying \eqref{eq:bumpupcenter} and \eqref{eq:bumpupbracket} yields $\bw\feq x_1x_1x_2x_1$, from which we can further rearrange to $\bw\feq x_1[x_1x_2x_1]\rfeq{eq:main2}[x_1\underline{x_2x_1x_2}]\rfeq{eq:leapfrog}x_1x_2x_2x_1$, which is in canonical form.  If $i_1=0$, then we either have $\bw\equiv x_1x_2x_1$ already in canonical form, or $i_2>1$ and then we have $\bw\equiv x_1x_2^{i_2}x_1\rfeq{eq:bumpupcenter}x_1x_2x_2x_1$, also in canonical form.
\end{proof}
Now let $\bw$ be a not necessarily connected word.  Then there is a unique decomposition into a product of connected subwords of maximal length and variables that appear just once in $\bw$; that is there is an $n$ such that $\bw\equiv \bw_1\bw_2\cdots\bw_n$ with each $\bw_i$ is either a letter appearing just once in $\bw$, or a connected word, and such that $\con(\bw_i)\cap \con(\bw_j)=\varnothing$ whenever $i\neq j$.  We say that $\bw$ is in \emph{canonical form} provided that each $\bw_i$ is in canonical form or is an individual letter.  It will be a consequence of the proof of Theorem~\ref{thm:xyx_yxy} below that distinct canonical forms do not form an identity following from $xyx= yxy$.

We consider the semigroup ${\bf C}$ given in Table \ref{table:C}.
\begin{table}
\begin{tabular}{c|ccccccccccc}
$*$  &0&1&2&3&4&5&6&7&8&9&10\\
\hline
0&0&0&0&0&0&0&0&0&0&0&0\\
1&0&1&6&5&7&5&6&7&0&10&10\\
2&0&0&4&8&5&0&0&0&9&5&0\\
3&0&0&0&3&0&0&0&0&0&0&0\\
4&0&0&5&9&0&0&0&0&5&0&0\\
5&0&0&0&5&0&0&0&0&0&0&0\\
6&0&0&7&0&5&0&0&0&10&5&0\\
7&0&0&5&10&0&0&0&0&5&0&0\\
8&0&0&0&8&0&0&0&0&0&0&0\\
9&0&0&0&9&0&0&0&0&0&0&0\\
10&0&0&0&10&0&0&0&0&0&0&0
\end{tabular}
\caption{The semigroup ${\bf C}$, a generator for $\ldbrack xyx=yxy\rdbrack$}\label{table:C}
\end{table}
The semigroup  ${\bf C}$ is isomorphic to the semigroup with presentation $\langle a,b,c\mid aa=a,b^4=0,cc=c,ba=cb=ca=abc=0,ab^3=b^3=b^3c=ac\rangle$.  To see this, first observe the relations in the presentation ensure that a nonzero product is always in nondecreasing alphabetical order, and then index laws $bbbb=0$ and $aa=a,cc=c$ and extra collapses $abc=0,ab^3=b^3=b^3c=ac$ ensure that there are exactly $11$ elements:
\[
0=abc,a=a^2,b,c=c^2,
bb,bbb=ac=abbb=bbbc=abbbc,
ab,abb,
bc,bbc,
abbc
\]
The map taking each element in this list to its numerical position in the list is an isomorphism onto ${\bf C}$ (that is, $0\mapsto 0$, $a\mapsto 1$, $b\mapsto 2$ and so on).  The semigroup ${\bf C}$ was found by hand: starting with the $3$-generated free algebra, successive quotients and subsemigroups were taken.  This led to a 16 element example.  In private communication, Edmond W.H. Lee observed that there were further quotients possible, and this eventually led to the current example.

To see that ${\bf C}\models xyx= yxy$, note that the only nonzero evaluations are $\theta(x)=\theta(y)\in\{1,2,3\}$ (in which case $\theta(xyx)=\theta(yxy)\in\{1,5,3\}$).  Note also that the subsemigroup on $\{1,3,5,0\}$ is the well-studied semigroup $A_0$, whose equational properties have some similarity to the those following from $xyx\feq yxy$.  
\begin{lemma}[Lee~\cite{lee}] \label{lem:A0}
Let $\bu\equiv \bu_1\cdots\bu_m$ and $\bv\equiv \bv_1\cdots\bv_n$ be a pair of words, where $\bu_1,\dots,\bu_m$ \up(and $\bv_1,\dots,\bv_n$ respectively\up) are pairwise disjoint words, each of which is either connected or a singleton.
Then $A_0\models u= v$ if and only if $m=n$ and $A_0\models \bu_i= \bv_i$.
Moreover,
\begin{enumerate}
\item if $\bu_i$ is a singleton, then $A_0\models \bu_i= \bv_i$ implies $\bu_i\equiv \bv_i$;
\item if $\bu_i$ is connected, then $A_0\models \bu_i= \bv_i$ if and only if $\con(\bu_i)=\con(\bv_i)$.
\end{enumerate}
\end{lemma}
\begin{theorem}\label{thm:xyx_yxy}
The variety defined by $xyx\feq yxy$ is generated by ${\bf C}$.
\end{theorem}
\begin{proof}
As ${\bf C}$ satisfies $xyx= yxy$, to show it generates the variety defined by $xyx\feq yxy$ it suffices to show that whenever  $\bu= \bv$ is an identity that does \emph{not} follow from $xyx= yxy$, then $\bu= \bv$ fails on ${\bf C}$.  By Lemma~\ref{lem:canonical2} we may assume without loss of generality that $\bu$ and $\bv$ are in canonical form.

As $\bu$ and $\bv$ are in canonical form, we may write
\begin{align*}
\bu&\equiv \bu_1\bu_2\bu_3\cdots\bu_m\\
\bv&\equiv \bv_1\bv_2\bv_3\cdots\bv_n
\end{align*}
where each $\bu_i$ and each $\bv_i$ are connected words in canonical form and such that $\con(\bu_i)\cap\con(\bu_j)=\varnothing$ for $i<j\leq m$ and $\con(\bv_i)\cap\con(\bv_j)=\varnothing$ for $i<j\leq n$.  Now $A_0\leq{\bf C}$, so Lemma~\ref{lem:A0} shows that we may assume that $n=m$ and $\con(\bu_i)=\con(\bv_i)$  for each $i=1,\dots,n$ (otherwise we have $A_0$ failing $\bu=\bv$ and we are done).

Now, as $\bu\neq\bv$ it follows that there is some $i$ with $\bu_i\neq \bv_i$.  Because of the definition of canonical form, and the fact that $\con(\bu_i)=\con(\bv_i)$, it follows that either there is a single variable $x$ such that $\bu_i\equiv x^j$ and $\bv_i\equiv x^k$ for some $j\neq k$ (with $j,k\leq 4$), or there are variables $x,y$ with $\bu_i\in\{xyx,xyyx\}$ and $\bv_i\in\{xyx,xyyx\}\backslash\{\bu_i\}$.  The second case may be mapped to the first of these cases by considering the substitution that fixes all variables but with $y\mapsto x$ (as $xyx\mapsto x^3$, while $xyyx\mapsto x^4$).  Without loss of generality then, let us assume $\bu_i\equiv x^j$, while $\bv_i\equiv x^k$ for $j<k\leq 4$.
Consider then the evaluation $\theta_1$ into ${\bf C}$ defined by
\[
\theta_1:z\mapsto
 \begin{cases}
1&\text{ if $z\in \con(\bu_1\cdots\bu_{i-1})$}\\
2&\text{ if $z= x$}\\
3&\text{ if $z\in \con(\bu_{i+1}\cdots\bu_{n})$.}
\end{cases}
\]
Now for $j=1,4$, we have $\theta_1(\bu) =0$, but $\theta_1(\bu)=10$ if $j=2$ and $\theta_1(\bu)=5$ if $j=3$.  Thus except in the case $\{j,k\}=\{1,4\}$, the substitution $\theta_1$ shows that $\bu=\bv$ fails on ${\bf C}$.  So now assume without loss of generality that $j=1$ (so that $k>1$)
Consider then the evaluation $\theta_2$ into ${\bf C}$ defined by
\[
\theta_1:z\mapsto
 \begin{cases}
1&\text{ if $z\in \con(\bu_1\cdots\bu_{i-1})$}\\
5&\text{ if $z= x$}\\
3&\text{ if $z\in \con(\bu_{i+1}\cdots\bu_{n})$.}
\end{cases}
\]
Then $\theta_2(\bu)=5$, while  $\theta_2(\bv)=0$.  Hence we have shown that ${\bf C}$ fails $\bu=\bv$, which completes the proof that the variety generated by ${\bf C}$ is the same as that defined by $xyx\feq yxy$.
\end{proof}

\begin{remark}\label{rem:xyxyxy}
The pseudovariety $\ldbrack xyx=yxy \rdbrack$ has precisely four maximal sub-pseudovarieties.  These are defined taking the law $xyx=yxy$ in conjunction with exactly one of the following laws: $x^4y^4=y^4x^4$, $x^4y^2z^4=x^4y^3z^4$, $x^4y^2z^4=x^4y^4z^4$, and  $x^4y^3z^4=x^4y^4z^4$.
\end{remark}
\begin{proof}
We may consider an equation  $\bu=\bv$ between distinct normal forms for $xyx=yxy$:
\begin{align*}
\bu&\equiv \bu_1\bu_2\bu_3\cdots\bu_m\\
\bv&\equiv \bv_1\bv_2\bv_3\cdots\bv_n
\end{align*}
(each $\bu_i$ and $\bv_i$ either a single variable or a connected component, in pairwise distinct alphabets).  Our goal is to deduce one of the four listed equations.  

Without loss of generality we may assume that $m,n\geq 3$ and that $\bu_1\equiv \bv_1$ and $\bu_m\equiv \bv_n$.  Indeed, if $a,b$ are letters not appearing in $\bu$ and $\bv$, then $a^4\bu b^4$ and $a^4\bv b^4$ are also distinct normal forms and $\bu=\bv\vdash a^4\bu b^4=a^4\bv b^4$.

If $\bu$ and $\bv$ have different alphabets (say, $y\in\con(\bv)\backslash\con(\bu)$), then by mapping $y\mapsto y^4$ and all other letters  to $x^4$ (and then simplifying to normal form) we obtain $\{\bu=\bv,xyx=yxy\}\vdash x^4=xy^2x$.  From this we obtain $x^4y^4=xy^2xy^4=x^4y^4x^4y^4=y^4x^4y^4x^4=yx^2yx^4=y^4x^4$.  So now we assume that $\bu$ and $\bv$ have the same alphabet $X$.  

Each connected component (or letter with single occurrence) has an alphabet that is a subset of $X$, and these subsets partition $X$.  If the partition of $X$ arising from $\bu$ is distinct from that arising from $\bv$ then we may deduce the law 
$x^4y^4=x y^2x$, from which $x^4y^4=y^4x^4$ again follows.  The same applies if the arising partitions coincide, but that the connected components appear in different order.  Thus we may assume now that $n=m\geq 3$ and $\con(\bu_i)=\con(\bv_i)$ for each $i=1,\dots,n$.  Let $i$ in $\{2,\dots,n-1\}$ be such that $\bu_i\not\equiv \bv_i$.  If $\bu_i$ has just one letter, then up to a change of latter names, $\bu'=y^j$ and $\bv'=y^k$, for some distinct $j,k\leq 4$ and letter $y$.  If $j$ or $k$ is $1$ then we may deduce $x^4y^3z^4=x^4y^4z^4$.  Otherwise, we have $\{j,k\}\in\{\{2,3\},\{2,4\},\{3,4\}\}$, from which one of the equations $x^4y^2z^4=x^4y^3z^4$, $x^4y^2z^4=x^4y^4z^4$, and $x^4y^3z^4=x^4y^4z^4$ are consequences of $\bu=\bv$.

Now assume that $|\con(\bu_i)|\geq 2$.  Given that $\con(\bu_i)=\con(\bv_i)$ but $\bu_i\not\equiv\bv_i$, it follows that $\bu_i=\bv_j$ is the equation $x_1x_2x_1=x_1x_2^2x_1$ (or reverse).  Then we obtain the consequence $x_1^4yzy x_2^4=x_1^4yz^2y x_2^4$.  But $x_1^4yzy x_2^4=x_1^4yz^2y x_2^4\vdash x^4y^3z^4=x^4y^4z^4$.

Finally we note that each of the four listed equations does not, in conjunction with $xyx=yxy$, imply any of the others.  It is possible to argue this syntactically, based on analysing the consequences of fully invariant congruences of the free algebra for $xyx=yxy$ on $3$ generators.   Alternatively, one may employ Mace4 again: for each of the four equations, there are three equations with which to compare.  Mace4 provides examples, of size between $4$ and $10$-elements, witnessing independence in each of the 12 cases.
\end{proof}

\subsection{Infinitely many atoms for $\Cnt(\PV)^+$.}
\begin{lemma}\label{lem:extralength}
Let $\bw_1=\bw_2$ satisfy the conditions of Proposition~\ref{prop:133}, and let $\{y_1,\dots,y_\ell\}\cap \con(\bw_1)=\varnothing$.  
Then $y_1\dots y_i\bw_1y_{i+1}\dots y_\ell=y_1\dots y_i\bw_2y_{i+1}\dots y_\ell$ satisfies the conditions of Proposition~\ref{prop:133}.  
Moreover, if $\ldbrack \bw_1=\bw_2\rdbrack$ is compact, then so is $\ldbrack y_1\dots y_i\bw_1y_{i+1}\dots y_\ell=y_1\dots y_i\bw_2y_{i+1}\dots y_\ell\rdbrack$.
\end{lemma}
\begin{proof}
The first statement is trivial.  For the second, observe that if $\ldbrack \bw_1=\bw_2\rdbrack$ is compact, then for some $m\in\mathbb{N}$, it is generated by the $m$-generated relatively free semigroup in the variety defined by $\bw_1=\bw_2$.  
We claim that $\ldbrack y_1\cdots y_i\bw_1y_{i+1}\cdots y_\ell=y_1\cdots y_i\bw_2y_{i+1}\cdots y_\ell\rdbrack$ is locally finite and generated by the $m+2$-generated relatively free algebra.  Let $F_j$ denote the relatively free semigroup for $\ldbrack y_1\cdots y_i\bw_1y_{i+1}\cdots y_\ell=y_1\cdots y_i\bw_2y_{i+1}\cdots y_\ell\rdbrack$ on $j$ free generators.

Now observe that if $\bu=\bv$ is a consequence of 
\[
 y_1\cdots y_i\bw_1y_{i+1}\cdots y_\ell=y_1\cdots y_i\bw_2y_{i+1}\cdots y_\ell,
\] then either $\bu\equiv \bv$, or $\bu\equiv\bp\bu'\bq$ and $\bv\equiv \bp\bv'\bq$, for some words $\bp,\bq,\bu',\bv'$ with $|\bp|=i$ and $|\bq|=\ell-i$, and where $\bu'=\bv'$ follows from $\bw_1=\bw_2$.  
This easily yields the fact that $\ldbrack y_1\cdots y_i\bw_1y_{i+1}\cdots y_\ell=y_1\cdots y_i\bw_2y_{i+1}\cdots y_\ell\rdbrack$ is locally finite provided $\ldbrack \bw_1=\bw_2\rdbrack$ is.

Next we show that $F_{m+1}$ generates the variety.  
For this we need to show that if $\bu=\bv$ does not follow from $y_1\cdots y_i\bw_1y_{i+1}\cdots y_\ell=y_1\cdots y_i\bw_2y_{i+1}\cdots y_\ell$, then $\bu=\bv$ fails on $F_{m+1}$.

If $\bu$ differs from $\bv$ within some prefix of length at most $i$, say $\bu\equiv\bu_1x\bu_2$ and $\bv\equiv\bu_1y\bu_2$ with $|\bu_1|<i$.  Then the substitution identifying all letters in $\con(\bu\bv)\backslash\{x\}$ with $y$ yields a failure of $\bu=\bv$ in $F_{2}\leq F_{m+1}$.  The case where $\bu$ differs from $\bv$ within some suffix of length at most $\ell-i$ is dual.

Now assume that $\bu$ and $\bv$ agree on the prefix of length $i$ and the suffix of length $\ell-i$.  It's possible the prefix overlaps with the suffix.  Because $\bu\not\equiv \bv$, this implies that $|\bu|\neq |\bv|$, with at least one of the $|\bu|,|\bv|<m+\ell$.  Then identifying all variables to $x$ yields $x^{|\bu|}=x^{|\bv|}$, which fails on $F_1$.  Thus we may assume that $\bu\equiv \bp\bu'\bq$, $\bv\equiv \bp\bv'\bq$, for some words $\bp,\bq,\bu',\bv'$ with $|\bp|=i$ and $|\bq|=\ell-i$, and where $\bu'=\bv'$ does \emph{not} follow from $\bw_1=\bw_2$.  
Let $\theta$ be an assignment from $\con(\bu'\bv')$ into $\{x_1,\dots,x_m\}$ for which $\theta(\bu')=\theta(\bv')$ does not follow from $\bw_1=\bw_2$; this exists because $\bw_1=\bw_2$ is generated by its $m$-generated free algebra.  Now extend $\theta$ to the other variables by  identifying all variables outside of $\{x_1,\dots,x_m\}$ to some $x\notin \{x_1,\dots,x_m\}$.  Then $\theta(\bu)=\theta(\bv)$ fails on $F_{m+1}$.
\end{proof}
It is easy to see that for fixed $\bw_1=\bw_2$, if the number $\ell$ in Lemma~\ref{lem:extralength} is increased, one obtains a different pseudovariety.  Then by Theorems~\ref{thm:xyx_xyy} and~\ref{thm:xyx_yxy}, one obtains infinitely many compact $\smi$s by using $xyx=yxy$ or $xyy=xyx$ for $\bw_1=\bw_2$.

We conclude with some open problems.
\begin{problem}\label{prob}
\begin{enumerate}
\item Describe all compact $\smi$ semigroup pseudovarieties.
\item If $S$ is a finite semigroup whose pseudovariety can be defined by a single equation, is it true that the variety of $S$ can be defined by a single equation?
\end{enumerate}
\end{problem}
In the direction of Problem \ref{prob}(1), a reasonable starting point would be to characterise which equations satisfying the conditions in Proposition \ref{prop:133} are compact; and are there any outside of those covered by Proposition \ref{prop:133}?  This falls within a more general problem, asking which finite systems of semigroup equations determine finitely generated varieties, and whether or not this is algorithmically solvable (the so-called ``reverse Tarski problem''; see O.~Sapir \cite{osap}).  A  further interesting intermediate problem  would be to examine which varieties determined by a single equation are finitely generated.   This leads to the second part of Problem \ref{prob}, which is a bounded version of the Eilenberg-Sch\"utzenberger problem (asking if a finite generator for a finitely based pseudovariety must generate a finitely based variety; see~\cite{eilsch}).  The Eilenberg-Sch\"utzenberger problem was solved positively for semigroup pseudovarieties by Mark Sapir \cite{msap2} but remains open for general algebras.  In connection with the present setting, observe that a $\smi$ pseudovariety must be definable (amongst finite semigroups) by a single equation.  Our arguments involve syntactic analysis of equational deductions, and would require adjustment if they were to cover any examples negatively answering Problem~\ref{prob}(2).  This problem also seems interesting for general algebras.  

\section*{Acknowledgments}
We thank Mark Sapir for several useful discussions, and Edmond W.\,H. Lee for bringing the article~\cite{SLsmall} to our attention, for the observations assisting in the reduction of the size of generator for the pseudovariety $\ldbrack xyx=yxy\rdbrack$, as well as suggesting the inclusion of Remarks \ref{rem:xyxxyy} and \ref{rem:xyxyxy}.  Lee, with the assistance of Jo\~{a}o Ara\'{u}jo, also noted an error in an early draft of Remark \ref{rem:xyxxyy}.

\separator
\def\malce{\mathbin{\hbox{$\bigcirc$\rlap{\kern-7.75pt\raise0,50pt\hbox{${\tt
  m}$}}}}}\def\cprime{$'$} \def\cprime{$'$} \def\cprime{$'$} \def\cprime{$'$}
  \def\cprime{$'$} \def\cprime{$'$} \def\cprime{$'$} \def\cprime{$'$}
  \def\cprime{$'$}

\end{document}